\documentclass[11pt]{amsart}
\usepackage{amssymb}
\usepackage{amsmath}
\usepackage{amsfonts}
\usepackage{graphicx}

\usepackage[total={17cm,22cm},top=2.5cm, left=2.3cm]{geometry}
\usepackage{hyperref}
    \usepackage{aeguill}
    \usepackage{type1cm}

\theoremstyle{plain}
\newtheorem{thm}{Theorem}[section]

\newtheorem{claim}[thm]{Claim}

\newtheorem{corollary}[thm]{Corollary}

\newtheorem{definition}[thm]{Definition}
\newtheorem{example}[thm]{Example}

\newtheorem{lemma}[thm]{Lemma}

\newtheorem{problem}[thm]{Problem}
\newtheorem{proposition}[thm]{Proposition}

\newtheorem{theorem}[thm]{Theorem}

\newtheorem{defn}[thm]{Definition}
\numberwithin{equation}{section}
\newcommand{\N}{\mathbb{N}}

\newcommand{\R}{\mathbb{R}}

\newcommand{\Rn}{\mathbb{R}^n}

\DeclareMathOperator{\lip}{Lip\,\!}

\usepackage[usenames,dvipsnames]{color}

\usepackage[dvipsnames]{xcolor}

\begin{document}

\title[Global geometry and $C^1$ convex extensions of $1$-jets]{Global geometry and $C^1$ convex extensions of $1$-jets}
\author{Daniel Azagra}
\address{ICMAT (CSIC-UAM-UC3-UCM), Departamento de An{\'a}lisis Matem{\'a}tico,
Facultad Ciencias Matem{\'a}ticas, Universidad Complutense, 28040, Madrid, Spain }
\email{azagra@mat.ucm.es}

\author{Carlos Mudarra}
\address{ICMAT (CSIC-UAM-UC3-UCM), Calle Nicol\'as Cabrera 13-15.
28049 Madrid, Spain}
\email{carlos.mudarra@icmat.es}

\date{June 29, 2017}


\thanks{D. Azagra was partially supported by Ministerio de Educaci\'on, Cultura y Deporte, Programa Estatal de Promoci\'on del Talento y su Empleabilidad en I+D+i, Subprograma Estatal de Movilidad. C. Mudarra was supported by Programa Internacional de Doctorado Fundaci\'on La Caixa--Severo Ochoa. Both authors partially suported by grant MTM2015-65825-P}


\begin{abstract}
Let $E$ be an arbitrary subset of $\R^n$ (not necessarily bounded), and $f:E\to\R$, $G:E\to\R^n$ be functions. We provide necessary and sufficient conditions for the $1$-jet $(f,G)$ to have an extension $(F, \nabla F)$ with $F:\R^n\to\R$ convex and $C^{1}$. Besides, if $G$ is bounded we can take $F$ so that $\textrm{Lip}(F)\lesssim \|G\|_{\infty}$. As an application we also solve a similar problem about finding convex hypersurfaces of class $C^1$ with prescribed normals at the points of an arbitrary subset of $\R^n$.
\end{abstract}

\maketitle

\section{Introduction and main results}

This paper concerns the following problem.

\begin{problem}\label{main problem}
Given $\mathcal{C}$ a differentiability class in $\R^n$, $E$ a subset of $\R^n$, and functions $f:E\to\R$ and $G:E\to\R^n$ , how can we decide whether there is a {\em convex} function $F\in\mathcal{C}$ such that $F(x)=f(x)$ and $\nabla F(x)=G(x)$ for all $x\in E$? 
\end{problem}

This is a natural question which we could solve in \cite{AzagraMudarra2017PLMS} in the case that $\mathcal{C}=C^{1, \omega}(\R^n)$, where $\omega: [0, \infty)\to [0, \infty)$ is a (strictly increasing and concave) modulus of continuity. A necessary and sufficient condition is that there exists a constant $M>0$ such that
$$
f(x) \geq f(y)+ \langle G(y), x-y \rangle + | G(x)-G(y)| \omega^{-1}\left( \frac{1}{2M} |G(x)-G(y)| \right) \,
\textrm{ for all } x, y\in E. \eqno(CW^{1,\omega})
$$
Very recently, some explicit formulas for such extensions have been found in \cite{DaniilidisLeGruyerEtAl} for the $C^{1,1}$ case, and more generally in \cite{AzagraMudarraExplicitFormulas} for the $C^{1, \omega}$ case when $\omega$ is a modulus of continuity with the additional property that $\omega(\infty)=\infty$; in particular this includes all the H\"older differentiability classes $C^{1, \alpha}$ with $\alpha\in (0, 1]$. Moreover, it can be arranged that 
$$
\sup_{x, y\in \R^n, \, x\neq y}\frac{|\nabla F(x)-\nabla F(y)|}{\omega(|x-y|)}\leq 8M
$$
(or even $\textrm{Lip}(\nabla F)\leq M$ in the $C^{1,1}$ case, that is to say when $\omega(t)=t$).

Besides the very basic character of Problem \ref{main problem}, there are other reasons for wanting to solve this kind of problem, as extension techniques for convex functions have natural applications in Analysis, Differential Geometry, PDE theory (in particular Monge-Amp\`ere equations), Economics, and  Quantum Computing. See the introductions of \cite{AzagraMudarra2017PLMS, GhomiJDG2001, MinYan} for background about convex extensions problems, and see \cite{BrudnyiShvartsman, Fefferman2005, Fefferman2006, FeffermanSurvey, FeffermanIsraelLuli2, FeffermanIsraelLuli1, Glaeser, JSSG, LeGruyer1} and the references therein for information about general Whitney extension problems.

Let $C^{1}_{\textrm{conv}}(\R^n)$ stand for the set of all functions $f:\R^n\to\R$ which are convex and of class $C^1$. In \cite{AzagraMudarra2017PLMS}, and for the class $\mathcal{C}=C^1(\R^n)$, we could only obtain a solution to Problem \ref{main problem} in the particular case that $E$ is a compact set. In this special situation the three necessary and sufficient conditions on $(f,G)$ that we obtained for $C^{1}_{\textrm{conv}}(\R^n)$ extendibility are:
$$
G \textrm{ is continuous, and }
\lim_{|z-y|\to 0^{+}}\frac{f(z)-f(y)-\langle G(y), z-y\rangle}{|z-y|}=0 \textrm{ uniformly on } E \eqno (W^1)
$$
(which is equivalent to Whitney's classical condition for $C^1$ extendibility), 
$$
f(x)-f(y)\geq \langle G(y), x-y\rangle \, \textrm{ for all } x,y\in E \eqno(C)
$$
(which ensures convexity), and
$$
f(x)-f(y)=\langle G(y), x-y\rangle \implies G(x)= G(y), \textrm{ for all } x,y\in E \eqno(CW^1)
$$
(which tells us that if two points of the graph of $f$ lie on a line segment contained in a hyperplane which we want to be tangent to the graph of an extension at one of the points, then our putative tangent hyperplanes at both points must be the same).
In fact, it is easy to see \cite[Remark 1.9]{AzagraMudarra2017PLMS} that continuity of $G$ plus conditions $(C)$ and $(CW^{1})$ imply Whitney's condition $(W^{1})$.

In \cite{AzagraMudarra2017PLMS} we also gave examples showing that the above conditions are no longer sufficient when $E$ is not compact (even if $E$ is an unbounded convex body). The reasons for this insufficiency can be mainly classified into two kinds of difficulties that only arise if the set $E$ is unbounded and $G$ is not uniformly continuous on $E$:
\begin{enumerate}
\item There may be no convex extension of the jet $(f, G)$ to the whole of $\R^n$.
\item Even when there are convex extensions of $(f, G)$ defined on all of $\R^n$, and even when some of these extensions are differentiable in some neighborhood of $E$, there may be no $C^1(\R^n)$ convex extension of $(f, G)$.
\end{enumerate}
The aim of this paper is to show how one can overcome these difficulties by adding new necessary conditions to $(W^{1}), (C), (CW^{1,1})$ in order to obtain a complete solution to Problem \ref{main problem} for the case that $\mathcal{C}=C^1(\R^n)$.

As is perhaps inevitable, our solution to Problem \ref{main problem} contains several technical conditions which may be quite difficult to grasp at a first reading. For this reason we will reverse the logical order of the exposition: we will start by providing some corollaries and examples. Only at last will the main theorem be stated.

The first kind of complication we have mentioned is well understood thanks to \cite{SchulzSchwartz}, and is not difficult to deal with: the requirement that
$$
\lim_{k\to\infty} \frac{\langle G(x_k), x_k \rangle - f(x_k)}{|G(x_k)|} = +\infty \textrm{
for every sequence } (x_k)_k \subset E \textrm{ with } \lim_{k\to\infty} | G(x_k) |=+\infty \eqno(EX)
$$
guarantees that there exist convex functions $\varphi:\R^n\to\R$ such that $\varphi_{|_E}=f$. In fact the extension $\varphi$ can be determined as the minimal convex extension of the jets, i.e., $\varphi(z)=\sup_{x\in E}\{ f(x)+\langle G(x), z-x\rangle\}$, and so that one has $G(x)\in\partial \varphi(x)$ (the subdifferential of $\varphi$ at $x$). See Lemma \ref{finitenessminimal} below.

The second kind of difficulty, however, is of a subtler geometrical character, and is related, on the one hand, to the rigid global behavior of convex functions (see Theorem \ref{decomposition theorem} below) and, on the other hand, to the fact that a differentiable (or even real-analytic) convex function $f:\R^n\to\R$ may have what one can call {\em corners at infinity}. As we indicated in $(2)$ above, there are examples of data $(E, f, G)$ which satisfy $(EX)$, $(W^1)$, $(C)$, and $(CW^{1})$ but which do not admit $C^1(\R^2)$ convex extensions. A prototypical instance is given in Example \ref{four examples}(4) below, or by the variant that we next formulate.
\begin{example}\label{example of a corner at infinity that cannot be extended}
{\em Consider $E=\{(x,y)\in\R^2 : y\leq \log |x|\}\cup \{(x,y)\in\R^2 : |x|\geq 1\}$, $f(x,y)=|x|$, $G(x,y)=(-1,0)$ if $x<0$, $G(x,y)=(1,0)$ if $x>0$.
Since the convex function $(x, y)\mapsto |x|$, together with its derivative, are smooth in $\R^{2}\setminus\{ x=0\}$ and extend the jet $(f, G)$ from $E$ to this region, it is clear that $(f, G)$ satisfies $(EX)$, $(W^1)$, $(C)$, and $(CW^{1})$ on $E$. We claim that there is no $C^1$ convex extension of $(f, G)$ to $\R^2$. The quickest and easiest way to prove this claim is just to apply Theorem \ref{MainTheoremWithThereExistsX} below. An intuitive geometrical argument that can be made rigorous is that in order to allow the {\em corner at infinity} that this jet has along the line $x=0$, any $C^1_{\textrm{conv}}(\R^2)$ extension of this datum would have to be {\em essentially coercive} in the direction of this line, but the requirement that $f(x,y)=|x|$ for $|x|\geq 1$ precludes the existence of any such extension. However, it is interesting to note that if we replace $E$ with $C=\{(x,y)\in\R^2 : y\leq \log |x|\}$ then the situation changes completely: there are $C^{1}_{\textrm{conv}}(\R^2)$ extensions of $(f, G)$ from $C$ to $\R^2$.}
\end{example}

In a short while we will be giving a precise meaning to the expressions {\em corner at infinity} and {\em essentially coercive}, but let us first ask ourselves this question:
what would appear to be a natural generalization of condition $(CW^{1})$ to the noncompact setting? In the absence of compactness it is natural to try to replace points 
by sequences, so as a first guess one is tempted to consider the following condition:
if $(x_k)_k, (z_k)_k$ are sequences in $E$ then
\begin{equation}\label{guessing conditions}
\lim_{k\to\infty}\left(f(x_k)-f(z_k)-\langle G(z_k), x_k-z_k\rangle\right)=0 \implies \lim_{k\to\infty}|G(x_k)-G(z_k)|=0.
\end{equation}
Unfortunately, if $E$ is unbounded this condition is not necessary for the existence of a convex function $F\in C^{1}(\R^n)$ such that $(F, \nabla F)=(f,G)$ on $E$, as the following example shows.
\begin{example}\label{essentially coercive corner at infinite in R2}
{\em Let $f:\R^2\to\R$ be defined by $f(x,y)=\sqrt{x^{2}+e^{-2y}}$. This is a real analytic strictly convex function on $\R^2$ (one can easily check that the Hessian $D^2 f$ is strictly positive everywhere). We have
$$
\nabla f(x,y)=\left( \frac{x}{\sqrt{x^{2}+e^{-2y}}}, -\frac{e^{-2y}}{\sqrt{x^{2}+e^{-2y}}}\right),
$$
and by considering the sequences 
$$
z_k=\left( \frac{1}{k}, k\right), \,\,\, x_k=(0, k),
$$
one easily sees that
$$
\lim_{k\to\infty}\left(f(x_k)-f(z_k)-\langle \nabla f(z_k), x_k-z_k\rangle\right)=0,
$$
and yet we have that
$$
\lim_{k\to\infty}|\nabla f(x_k)-\nabla f(z_k)|=1\neq 0
$$ }
\end{example}
(according to Definition \ref{definition of corner at infinity} below this means that the jet $(f, \nabla f)$ has a {\em corner at infinity} along the line $x=0$).
So our first guess turned out to be wrong, and we have to be more careful. In view of the above example, and at least if we are looking for extensions $(F, \nabla F)$ with $F\in C^{1}(\R^n)$ convex and essentially coercive (that is, $C^1$ convex extensions $F(x)$ which, up to a linear perturbation, tend to $\infty$ as $|x|$ goes to infinity), it could make sense to restrict condition \eqref{guessing conditions} to sequences $(x_k)_k$ which are bounded. On the other hand, if $(G(z_k))_k$ is not bounded as well, then by using condition $(EX)$, up to extracting a subsequence, we would have
$$
\lim_{k\to\infty} \frac{\langle G(z_k), z_k \rangle - f(z_k)}{|G(z_k)|} = \infty,
$$
hence
$$
\langle G(z_k), z_k \rangle - f(z_k)= M_k |G(z_k)|, \,\,\, \textrm{ with } \lim_{k\to\infty}M_k=\infty,
$$
and it follows that
\begin{eqnarray*}
& & f(x_k)-f(z_k)-\langle G(z_k), x_k-z_k\rangle= f(x_k)-f(z_k)+\langle G(z_k), z_k\rangle -\langle G(z_k), x_k\rangle \geq \\
& & f(x_k)+ \left( M_k-|x_k|\right) |G(z_k)|\to\infty
\end{eqnarray*}
(because $(f(x_k))_k$ and $(x_k)_k$ are bounded and $M_k\to\infty$). Thus we have learned that we cannot have
$$
\lim_{k\to\infty}\left(f(x_k)-f(z_k)-\langle G(z_k), x_k-z_k\rangle\right)=0
$$
unless $(G(x_k))_k$ is bounded. An educated guess for a good substitute of $(CW^1)$ could then be to require that
\begin{equation}\label{guessing conditions 2}
\lim_{k\to\infty}\left(f(x_k)-f(z_k)-\langle G(z_k), x_k-z_k\rangle\right)=0 \implies \lim_{k\to\infty}|G(x_k)-G(z_k)|=0
\end{equation}
\hspace{2.2cm} {\em for all sequences $(x_k)_k$ and $(z_k)_k$ in $E$ such that $(x_k)_k$ and $(G(z_k))_k$ are bounded.}

This new condition can be checked to be necessary for the existence of a function $F$ which solves our problem. Now, if we add \eqref{guessing conditions 2} to $(EX)$ and $(C)$, will this new set of conditions be sufficient as well? The answer to this question depends on how large the set $\textrm{span}\{G(x)-G(y) : x, y\in E\}$ is. If this set coincides with $\R^n$ then those conditions are sufficient: this is the content of the following easy\footnote{But especially useful, as for generic initial data $(E, f, G)$ one has
$\textrm{span}\{G(x)-G(y) : x, y\in E\}=\R^n$} consequence of the main result of this paper.
On the other hand, if we do not have $\textrm{span}\{G(x)-G(y) : x, y\in E\}=\R^n$ then we already know by Example \ref{example of a corner at infinity that cannot be extended} that our problem will be more difficult to tackle.

\begin{corollary}\label{MainCorollaryWithEssentiallyCoerciveExtensions}
Fix an arbitrary subset $E$ of $\Rn$ and two functions $f:E \to \R, \: G: E \to \Rn$.
Suppose that $\textrm{span}\{G(x)-G(y) : x, y\in E\}=\R^n$. Then there exists a convex function $F:\R^n\to\R$ of class $C^1$ with $F_{|_E}=f$ and $(\nabla F)_{|_E}=G$ if and only if the following conditions are satisfied:
\begin{itemize}
\item[$(i)$] $G$ is continuous and $f(x)\geq f(y)+\langle G(y), x-y\rangle$ for all $x, y\in E$.
\item[$(ii)$] If $(x_k)_k \subset E$ is a sequence for which $\lim_{k\to\infty} | G(x_k) |=+\infty,$ then
$$
\lim_{k\to\infty} \frac{\langle G(x_k), x_k \rangle - f(x_k)}{|G(x_k)|} = +\infty.
$$
\item[$(iii)$] If $(x_k)_k, (z_k)_k$ are sequences in $E$ such that $( x_k)_k$ and $(G(z_k))_k$ are bounded, and
$$
 \lim_{k\to\infty} \left( f(x_k)-f(z_k)- \langle G(z_k), x_k-z_k \rangle \right) = 0,
$$
then $\lim_{k\to\infty} | G(x_k)-G(z_k) | = 0$.
\end{itemize}
Moreover, whenever these conditions are satisfied, the extension $F$ can be taken to be {\em essentially coercive}.
\end{corollary}

Here, by saying that $F$ is {\em essentially coercive} we mean that there exists a linear function $\ell:\R^n\to\R$ such that 
$$
\lim_{|x|\to\infty}\left( F(x)-\ell(x)\right)=\infty.
$$

Let us mention that the above corollary is applied in \cite{AzagraHajlasz} to show that a convex function $f:\R^n\to\R$ has a
Lusin property of type $C^{1}_{\textrm{conv}}(\R^n)$ (meaning that for every $\varepsilon>0$ there exists a convex function $g\in C^{1}(\R^n)$ such that $\mathcal{L}^{n}\left(\{x\in\R^n : f(x)\neq g(x)\}\right)<\varepsilon$, where $\mathcal{L}^n$ denotes Lebesgue's measure) if and only if either $f$ is essentially coercive or else $f$ is already $C^1$ (in which case taking $g=f$ is the only possible option). 

In order to quickly understand Corollary \ref{MainCorollaryWithEssentiallyCoerciveExtensions}, instead of looking at the rather technical proof of Theorem \ref{MainTheorem1WithPThenP} below we recommend reading the proof of \cite[Theorem 1.10]{AzagraMudarra2017PLMS}, which can be easily adapted to produce a simpler proof of Corollary \ref{MainCorollaryWithEssentiallyCoerciveExtensions}.

\medskip

By comparing Examples \ref{example of a corner at infinity that cannot be extended} and \ref{essentially coercive corner at infinite in R2} with Corollary \ref{MainCorollaryWithEssentiallyCoerciveExtensions} we may arrive at a remarkable conclusion: our given jet $(f, G)$ may well have some {\em  corners at infinity} and, for $C^1$ convex extension purposes, that will not matter at all as long as $(f, G)$ forces all possible convex extensions to be essentially coercive (equivalently, as long as $\textrm{span}\{G(x)-G(y) : x, y\in E\}=\R^n$). But, if the given datum presents some corners at infinity and does not force essential coercive in the directions of those corners, then we will have to be more careful, as $C^{1}_{\textrm{conv}}(\R^n)$ extensions may not exist in this case.

Let us now explain what we mean by a jet having a corner at infinity. 
\begin{definition}\label{definition of corner at infinity}
{\em Let $X$ be a proper linear subspace of $\R^n$ and let us denote by $X^\perp$ its orthogonal complement. We say that a jet $(f, G):E\subset\R^n\to\R\times \R^n$ has a corner at infinity in a direction of $X^{\perp}$ provided that there exist two sequences $(x_k)_k, (z_k)_k$ in $E$ such that, if $P_{X}:\R^n\to X$ denotes the orthogonal projection, we have that
$(P_X(x_k))_k$ and $(G(z_k))_k$ are bounded, $\lim_{k\to\infty}|x_k|=\infty$, 
$$
\lim_{k\to\infty} \left( f(x_k)-f(z_k)- \langle G(z_k), x_k-z_k \rangle \right) = 0,
$$
and yet
$$
\limsup_{k\to\infty} | G(x_k)-G(z_k) | > 0.
$$
We will also say that the jet $(f, G)$ has a corner at infinity in the direction of the line $\{tv : t\in\R\}$ (where $v\in\R^{n}\setminus\{0\}$) provided that there exist sequences $(x_k)_k, (z_k)_k$ satisfying the above properties with $P_X$ being the orthogonal projection onto the hyperplane $X$ perpendicular to $v$.}
\end{definition}
For instance, the function $f(x,y)$ of Example \ref{essentially coercive corner at infinite in R2}, and its gradient, when restricted to the sequences $(x_k)_k$, $(z_k)_k$ defined there, give an instance of a jet that has a corner at infinity directed by the line $x=0$. Of course, the pair $(f, \nabla f)$, unrestricted, provides another instance. In this case it is natural to say that the function $f$ itself has a corner at infinity. More pathological examples can be given in higher dimensions: for instance, if $1\leq k \leq n$ then
\begin{equation}\label{example of corner function 2}
f(x_1, x_2, \ldots, x_n)=\sqrt{\sum_{j=1}^{k}x_{j}^2 +\sum_{j=k+1}^{n} e^{-2x_j}}, \quad x=(x_1,\ldots,x_n)\in \R^n
\end{equation}
is a convex function of class $C^\infty$ with strictly positive Hessian at every point, which has a corner at infinity in the direction of $e_j$ for every $j=k+1,\ldots,n$, and which is essentially coercive. On the other hand, if $n \geq 3$ and $2 \leq k <n$, then
\begin{equation}\label{example of corner function 3}
g(x_1, \ldots, x_n)=\sqrt{x_{1}^2 +\sum_{j=2}^{k} e^{-2x_j}} , \quad x=(x_1,\ldots,x_n)\in \R^n.
\end{equation}
is convex and of class $C^\infty,$ $g$ has a corner at infinity in the direction of $e_j,$ for every $j=2,\ldots,k,$ and $g$ is not essentially coercive. Nevertheless $g$ is {\em
essentially $k$-coercive} (meaning that $g$
can be written as $g=c\circ P$, where $P$ is the orthogonal projection onto a
$k$-dimensional subspace of $X$ of $\R^n$ and $c:X\to\R$ is essentially coercive).

In general it can be shown (and in fact this is a consequence of our main results) that the presence of a corner at infinity in the graph of a differentiable convex function $f:\R^n\to\R$ forces essential $k$-coercivity of $f$, for some $k\geq 2$, in a subspace of directions containing the directions of the corner.

We will not explicitly use the notion of corner at infinity in our proofs. Our reasons for introducing these objects are the facts that: 1) one way or another, corners at infinity will be to blame for most of the predicaments and technicalities involved in any attempt to solve Problem \ref{main problem} for $\mathcal{C}=C^1(\R^n)$; and 2) we firmly believe that the reader will be more able to understand the statements and proofs of the following results once he has been acquainted with this notion. As a matter of fact, the most technical conditions of Theorems \ref{MainTheoremWithThereExistsX} and \ref{MainTheorem1WithPThenP} below can be rephrased more intuitively in terms of corners at infinity and essential coercivity of data in the directions of those corners.

\medskip

Unfortunately Corollary \ref{MainCorollaryWithEssentiallyCoerciveExtensions} does not provide a characterization of the $1$-jets which admit essentially coercive $C^1$ convex extensions. This is due to the fact that a jet $(f, G)$ defined on a set $E$ may admit such an extension and yet $\textrm{span}\{G(x)-G(y): x, y\in E\}\neq\R^n$, as shown by the trivial example of the jet $(f_0, G_0)$ with $E_0=\{0\}\subset\R^2$, $f_0(0)=0$, $G_0(0)=0$, which admits a $C^1$ convex and coercive extension given by $(F_0, \nabla F_0)$, where $F_0(x,y)=x^2 +y^2$.

Of course, a $C^1$ convex extension problem for a given $1$-jet $(f, G)$ may have solutions which are not essentially coercive; in fact it may happen that none of its solutions are essentially coercive. A sister of Corollary \ref{MainCorollaryWithEssentiallyCoerciveExtensions} which provides a more general, but still partial  solution to Problem \ref{main problem}, is the following.

\begin{corollary}\label{corollary with Y equals span G(x)-G(y)}
Given an arbitrary subset $E$ of $\Rn$ and two functions $f:E \to \R, \: G: E \to \Rn$, assume that the following conditions are satisfied:
\begin{itemize}
\item[$(i)$] $G$ is continuous and $f(x)\geq f(y)+\langle G(y), x-y\rangle$ for all $x, y\in E$.
\item[$(ii)$] If $(x_k)_k \subset E$ is a sequence for which $\lim_{k\to\infty} | G(x_k) |=+\infty,$ then
$$
\lim_{k\to\infty} \frac{\langle G(x_k), x_k \rangle - f(x_k)}{|G(x_k)|} = +\infty.
$$
\item[$(iii)$] Let $P=P_Y:\R^n\to\R^n$ be the orthogonal projection onto $Y:=\textrm{span}\{G(x)-G(y)\: : \: x,y \in E\}$. If $(x_k)_k, (z_k)_k$ are sequences in $E$ such that $(P (x_k))_k$ and $(G(z_k))_k$ are bounded and
$$
\lim_{k\to\infty} \left( f(x_k)-f(z_k)- \langle G(z_k), x_k-z_k \rangle \right) = 0,
$$
then $\lim_{k\to\infty} | G(x_k)-G(z_k) | = 0$.
\end{itemize}
Then there exists a convex function $F:\R^n\to\R$ of class $C^1$ such that $F_{|_E}=f$ and $(\nabla F)_{|_E} =G$.
\end{corollary}
Condition $(iii)$ of the above corollary can be intuitively rephrased by saying that: 1) our jet satisfies a natural generalization of condition $(CW^1)$; and 2) $(f, G)$ cannot have corners at infinity in any direction contained in the orthogonal complement of the subspace $Y=\textrm{span}\{G(x)-G(y)\: : \: x,y \in E\}$.

It could be natural to hope for the conditions of Corollary \ref{corollary with Y equals span G(x)-G(y)} to be necessary as well, thus providing a nice characterization of those $1$-jets which admit $C^1$ convex extensions. Unfortunately the solution to Problem \ref{main problem} is necessarily more complicated, as the following example shows. 
\begin{example}
{\em Let $E_1=\{(x,y)\in\R^2 : y=\log |x|, |x|\in\N \cup\{\frac{1}{n}  : n\in\N\}\}$, $f_1(x,y)=|x|$, $G_1(x,y)=(-1,0)$ if $x<0$, $G_1(x,y)=(1,0)$ if $x>0$. In this case we have $Y_1:=\textrm{span}\{G_1(x,y)-G_1(x',y') \: : \: (x,y), (x',y') \in E_1 \}=\R\times\{0\}$, and it is easily seen that condition $(iii)$ is not satisfied. However, it is not difficult to check that, for $\varepsilon>0$ small enough, if we set $E_{1}^{*}=E_1\cup \{(0,1)\}$, $f_{1}^{*}=f_1$ on $E_1$, $f_{1}^{*}(0,1)=\varepsilon$, $G_{1}^{*}=G_1$ on $E_1$, and $G_{1}^{*}(0,1)=(0, \varepsilon)$, then Corollary \ref{MainCorollaryWithEssentiallyCoerciveExtensions} implies that the problem of finding a $C^{1}$ convex extension of the jet $(f_{1}^{*}, G_{1}^{*})$ does have a solution, and therefore the same is true of the jet $(f_1, G_1)$. }
\end{example}

This example shows that in some cases the $C^{1}$ convex extension problem for a $1$-jet $(f, G)$ may be geometrically underdetermined in the sense that we may not have been given enough differential data so as to have condition $(iii)$ of the above corollary satisfied with $Y=\textrm{span}\{G(x)-G(y) \: : \: x,y\in E \}$, and yet it may be possible to find a few more jets $(\beta_j, w_j)$ associated to finitely many points $p_j\in\R^n\setminus \overline{E}$, $j=1, \ldots, m$, so that, if we define $E^{*}=E\cup\{p_1, \ldots , p_{m}\}$ and extend the functions $f$ and $G$ from $E$ to $E^{*}$ by setting
\begin{equation}\label{definitionnewdata}
f(x_j):=\beta_j, \quad G(p_j):= w_j \quad \text{for} \quad j=1, \ldots, m,
\end{equation}
then the new extension problem for $(f, G)$ defined on $E^{*}$ does satisfy condition $(iii)$ of Corollary \ref{corollary with Y equals span G(x)-G(y)}. Notice that, the larger $Y$ grows, the weaker condition $(iii)$ of Corollary \ref{corollary with Y equals span G(x)-G(y)} becomes.

We are now prepared to state a first version of our main result.

\begin{theorem}\label{MainTheoremWithThereExistsX}
Given an arbitrary subset $E$ of $\Rn$ and two functions $f:E \to \R, \: G: E \to \Rn$, the following is true. There exists a convex function $F:\R^n\to\R$ of class $C^1$ such that $F_{|_E}=f$, and $(\nabla F)_{|_E} =G$,  if and only if the following conditions are satisfied.
\begin{itemize}
\item[$(i)$] $G$ is continuous and $f(x)\geq f(y)+\langle G(y), x-y\rangle$ for all $x, y\in E$.
\item[$(ii)$] If $(x_k)_k \subset E$ is a sequence for which $\lim_{k\to\infty} | G(x_k) |=+\infty,$ then
$$
\lim_{k\to\infty} \frac{\langle G(x_k), x_k \rangle - f(x_k)}{|G(x_k)|} = +\infty.
$$
\item[$(iii)$] Let $Y:=\textrm{span}\{G(x)-G(y) \: : \: x,y\in E\}.$ There exists a linear subspace $X\supseteq Y$ such that, either $Y=X$, or else, if we denote $k= \dim Y,\:d= \dim X$ and $P_X: \R^n \to\R^n$ is the orthogonal projection from $\R^n$ onto $X$, there exist points $p_1, \ldots, p_{d-k} \in \R^n \setminus \overline{E}$, numbers $\beta_1,\ldots, \beta_{d-k} \in \R$, and vectors $w_1, \ldots, w_{d-k} \in \R^n$  such that:
\begin{itemize}
\item[(a)] $X= \textrm{span} \left( \lbrace u-v \: : \: u,v\in G(E)\cup \lbrace w_1,\ldots,w_{d-k} \rbrace \rbrace \right).$
\item[(b)] $\beta_j > \max_{1 \leq i \neq j \leq d-k} \lbrace \beta_i + \langle w_i, p_j -p_i \rangle \rbrace$ for all $1 \leq j \leq d-k.$
\item[(c)] $\beta_j > \sup_{z\in E, \: |G(z)| \leq N} \lbrace f(z)+ \langle G(z), p_j -z \rangle \rbrace$ for all $1 \leq j \leq d-k$ and $N \in \N.$
\item[(d)] $\inf_{x\in E, \: |P_X(x)| \leq N}\lbrace f(x) - \max_{1 \leq j \leq d-k} \lbrace \beta_j + \langle w_j, x-p_j\rangle \rbrace \rbrace >0$ for all $N \in \N.$
\end{itemize}
\item[$(iv)$] If $X$ and $P_X$ are as in $(iii)$, and $(x_k)_k, (z_k)_k$ are sequences in $E$ such that $(P_X(x_k))_k$ and $(G(z_k))_k$ are bounded and
$$
\lim_{k\to\infty} \left( f(x_k)-f(z_k)- \langle G(z_k), x_k-z_k \rangle \right) = 0,
$$
then $\lim_{k\to\infty} | G(x_k)-G(z_k) | = 0$.
\end{itemize}
\end{theorem}

As we see, the difference between Theorem \ref{MainTheoremWithThereExistsX} and Corollary \ref{corollary with Y equals span G(x)-G(y)} is in the technical condition $(iii)$, which
can be informally summed up by saying that, whenever the jets $(f(x), G(x))$, $x\in E$, do not provide us with enough differential data so that condition $(iii)$ of Corollary \ref{corollary with Y equals span G(x)-G(y)} holds, there is enough room in $\R^n\setminus \overline{E}$ to add finitely many new jets $(\beta_j, w_j)$, associated to new points $p_j$, $j=1, \ldots , d-k$, in such a way that the new extension problem does satisfy the conditions of Corollary \ref{corollary with Y equals span G(x)-G(y)}. This condition also tells us that the new extension problem will be one for which, even though there may be corners at infinity, those corners will necessarily be directed by subspaces which are contained in the span of the putative derivatives, and the new data will force essential coercivity of all possible extensions in the directions of the corners.

Later on we will show that, in the particular case that $G$ is bounded (and so we may expect to find an $F$ with a bounded gradient), these complicated conditions about compatibility of the old and new data admit a much nicer geometrical reformulation, see Theorem \ref{maintheoremlipschitzc1} below.

Let us consider some examples that will hopefully offer further clarification of these comments.

\begin{example}\label{four examples}
{\em Consider the following $1$-jets $(f_j, G_j)$ defined on subsets $E_j$ of $\R^n$:
\begin{enumerate}
\item  $E_1=\{(x,y)\in\R^2 : y=\log |x|, |x|\in\N \cup\{\frac{1}{n}  : n\in\N\}\},$ $f_1(x,y)=|x|$, $G_1(x,y)=(-1,0)$ if $x<0$, $G_1(x,y)=(1,0)$ if $x>0$.
\item  $E_2=\{(x,y)\in\R^2 : y=\log |x|, |x| \in\N \cup\{\frac{1}{n} : n\in\N\}\},$ $f_2=\varphi$, $G_{2}=\nabla\varphi$, where $\varphi(x,y)=\sqrt{x^2+ e^{-2y}}$.
\item $E_3=\{(x,y, z)\in\R^3 : z=0, y=\log |x|, |x|\in\N \cup\{ \frac{1}{n} : n\in\N\}\},$  $f_3=\varphi$, $G_{3}=\nabla\varphi$, where $\varphi(x,y,z)=\sqrt{x^2+ e^{-2y}}$.
\item $E_4=E_1\cup \{(x,y)\in\R^2 : |x|\geq 1\}$, $f_4(x,y)=|x|$, $G_4(x,y)=(-1,0)$ if $x<0$, $G_4(x,y)=(1,0)$ if $x>0$.
\end{enumerate}
Then one can check that:
\begin{enumerate}
\item[(i)] For the jet $(f_1, G_1)$, and with the notation of Theorem \ref{MainTheoremWithThereExistsX}, we have $Y=\R\times\{0\}$, but the smallest possible $X$ we can take is $X=\R^2$ (and all possible extensions $F$ must be essentially coercive on $\R^2$).
\item[(ii)] For the jet $(f_2, G_2)$ we have $Y=\R^2$, and all possible extensions $F$ must be essentially coercive on $\R^2$.
\item[(iii)] For the jet  $(f_3, G_3)$ we have $Y=\R^2\times\{0\}$, and we can take either $X=Y$ or $X=\R^3$.
\item[(iv)] For the jet $(f_4, G_4)$ we have $Y=\R\times\{0\}$, but one cannot apply Theorem \ref{MainTheoremWithThereExistsX} with any $X$. There exists no $F\in C^{1}_{\textrm{conv}}(\R^2)$ such that $(F, \nabla F)$ extends $(f_4, G_4)$.
\end{enumerate}
}
\end{example}

Even though Theorem \ref{MainTheoremWithThereExistsX} fully solves Problem \ref{main problem}, an important question\footnote{Coercitivity of a convex function may be relevant or even essential to a number of possible applications, e.g. in PDE theory.} remains open: how can we characterize those $1$-jets $(f, G)$ such that there exists an essentially coercive convex function $F\in C^{1}(\R^n)$ so that $(F, \nabla F)$ extends $(f, G)$? The answer is: those jets are the jets which satisfy the conditions of Theorem \ref{MainTheoremWithThereExistsX} with $X=\R^n$. More generally, one could ask for $C^1$ convex extensions with prescribed global behavior (meaning extensions which are essentially coercive only in some directions, and affine in others). This ties in with a question which will be extremely important in our proofs: what is the global geometrical shape of the $C^1$ convex extension we are trying to build? 

In this regard, it will be convenient for us to state a refinement of Theorem \ref{MainTheoremWithThereExistsX} which characterizes the set of $1$-jets admitting $C^1$ convex extensions with a prescribed global behavior, and which requires our introducing some definitions and notation.

\begin{defn}
{\em Let $Z$ be a real vector space, and $P:Z\to X$ be the orthogonal projection onto a subspace $X\subseteq Z$.
We will say that a function $f$ defined on a subset $E$ of $Z$ is {\em essentially $P$-coercive} provided that there exists a linear function $\ell:Z\to\R$ such that for every sequence $(x_k)_k\subset E$ with $\lim_{k\to\infty}|P(x_k)|=\infty$ one has
$$
\lim_{k\to\infty}\left(f-\ell\right)(x_k)=\infty.
$$
We will say that $f$ is {\em essentially coercive} whenever $f$ is essentially $I$-coercive, where $I:Z\to Z$ is the identity mapping.

If $X$ is a linear subspace of $\R^n$, we will denote by $P_{X}:\R^n\to X$ the orthogonal projection, and we will say that $f:E\to \R$ is {\em coercive in the direction of $X$} whenever $f$ is $P_{X}$-coercive. 

We will also denote by $X^{\perp}$ the orthogonal complement of $X$ in $\R^n$. For a subset $V$ of $\R^n$, $\textrm{span}(V)$ will stand for the linear subspace spanned by the vectors of $V$. 
}
\end{defn}

In \cite{Azagra} essentially coercive convex functions were called {\em properly convex}, and some approximation results, which fail for general convex functions, were shown to be true for this class of functions.
The following result was also implicitly proved in \cite[Lemma 4.2]{Azagra}. Since this will be a very important tool in the statements and proofs of all the results of the present paper,  and because we have introduced new terminology and added conclusions, we will provide a self-contained proof in Section 2 for the readers' convenience.

\begin{theorem}\label{decomposition theorem}
For every convex function $f:\R^n\to\R$ there exist a unique linear subspace $X_f$ of $\R^n$, a unique vector $v_{f}\in X_{f}^{\perp}$, and a unique essentially coercive function $c_{f}:X_f\to\R$ such that $f$ can be written in the form
$$
f(x)=c_{f}(P_{X_{f}}(x)) +\langle v_{f}, x\rangle \textrm{ for all } x\in\R^n.
$$
Moreover, if $Y$ is a linear subspace of $\R^n$ such that $f$ is essentially coercive in the direction of $Y$, then $Y\subseteq X_f$.
\end{theorem}

The following Proposition shows that the directions $X_f$ given by these decompositions are stable by approximation.

\begin{proposition}
With the notation of the preceding theorem, if $f, g:\R^n\to\R$ are convex functions and $A$ is a positive number such that
$
f(x)\leq g(x)+A
$ for all $x\in \R^n$,
then $X_f\subseteq X_g$.

In particular, if $|f-g|\leq A$ then $X_f=X_g$.
\end{proposition}
\begin{proof}
The inequality $f(x)\leq g(x)+A$ and the essential coercivity of $f$ in the direction $X_{f}$ implies that $g$ is essentially coercive in the direction $X_f$. Then $X_f\subseteq X_g$ by the last part of Theorem \ref{decomposition theorem}.
\end{proof}

We are finally ready to state the announced refinement of Theorem \ref{MainTheoremWithThereExistsX} which characterizes precisely what $1$-jets $(f,G)$ admit extensions $(F, \nabla F)$ such that $F\in C^{1}_{\textrm{conv}}(\R^n)$ and $X_{F}$ coincides with a prescribed linear subspace $X$ of $\R^n$.

\begin{theorem}\label{MainTheorem1WithPThenP}
Given an arbitrary subset $E$ of $\Rn$, a linear subspace $X\subset\R^n$, the orthogonal projection $P:=P_{X}:\R^n\to X$, and two functions $f:E \to \R, \: G: E \to \Rn$, the following is true. There exists a convex function $F:\R^n\to\R$ of class $C^1$ such that $F_{|_E}=f$, $(\nabla F)_{|_E} =G$, and $X_{F}=X$, if and only if the following conditions are satisfied.
\begin{itemize}
\item[$(i)$] $G$ is continuous and $f(x)\geq f(y)+\langle G(y), x-y\rangle$ for all $x, y\in E$.
\item[$(ii)$] If $(x_k)_k \subset E$ is a sequence for which $\lim_{k\to\infty} | G(x_k) |=+\infty,$ then
$$
\lim_{k\to\infty} \frac{\langle G(x_k), x_k \rangle - f(x_k)}{|G(x_k)|} = +\infty.
$$
\item[$(iii)$] $Y:=\textrm{span}\left(\{G(x)-G(y) : x, y\in E\}\right)\subseteq X$.
\item[$(iv)$] If $Y \neq X$ and we denote $k= \dim Y$ and $d= \dim X$, there exist points $p_1, \ldots, p_{d-k} \in \R^n \setminus \overline{E}$, numbers $\beta_1,\ldots, \beta_{d-k} \in \R$, and vectors $w_1, \ldots, w_{d-k} \in \R^n$  such that:
\begin{itemize}
\item[(a)] $X= \textrm{span} \left( \lbrace u-v \: : \: u,v\in G(E)\cup \lbrace w_1,\ldots,w_{d-k} \rbrace \rbrace \right).$
\item[(b)] $\beta_j > \max_{1 \leq i \neq j \leq d-k} \lbrace \beta_i + \langle w_i, p_j -p_i \rangle \rbrace$ for all $1 \leq j \leq d-k.$
\item[(c)] $\beta_j > \sup_{z\in E, \: |G(z)| \leq N} \lbrace f(z)+ \langle G(z), p_j -z \rangle \rbrace$ for all $1 \leq j \leq d-k$ and $N \in \N.$
\item[(d)] $\inf_{x\in E, \: |P(x)| \leq N}\lbrace f(x) - \max_{1 \leq j \leq d-k} \lbrace \beta_j + \langle w_j, x-p_j\rangle \rbrace \rbrace >0$ for all $N \in \N.$
\end{itemize}
\item[$(v)$] If $(x_k)_k, (z_k)_k$ are sequences in $E$ such that $(P (x_k))_k$ and $(G(z_k))_k$ are bounded and
$$
\lim_{k\to\infty} \left( f(x_k)-f(z_k)- \langle G(z_k), x_k-z_k \rangle \right) = 0,
$$
then $\lim_{k\to\infty} | G(x_k)-G(z_k) | = 0$.
\end{itemize}
\end{theorem}
In particular, by considering the case that $X=\R^n$, we obtain a characterization of the $1$-jets which admit $C^1$ convex extensions which are essentially coercive in $\R^n$.

It is clear that Theorem \ref{MainTheoremWithThereExistsX} and Corollaries \ref{MainCorollaryWithEssentiallyCoerciveExtensions} and \ref{corollary with Y equals span G(x)-G(y)} are immediate consequences of the above theorem.
The proof of Theorem \ref{MainTheorem1WithPThenP} will be given in Sections 3 and 4. 

In the special case that the function $G$ of the above Theorem is bounded, one should expect to find Lipschitz convex functions $F\in C^{1}(\R^n)$ such that $(F, \nabla F)$ extends $(f, G)$ and $\textrm{Lip}(F)\lesssim \|G\|_{\infty}$. Notice that this kind of control of $\textrm{Lip}(F)$ in terms of $\sup_{y\in E} |G(y)|$ solely cannot be obtained, in general, for nonconvex jets, but it is possible in the convex case, at least when $E$ is compact; see the comments after \cite[Theorem 1.10]{AzagraMudarra2017PLMS}. The next result tells us that this is indeed feasible, and moreover shows that the technical conditions of $(iv)$ in Theorem \ref{MainTheorem1WithPThenP} can be replaced (just in this Lipschitz case) by a nicer geometric condition which tells us that the complement of the closure of $E$ in $\R^n$ contains the union of a certain finite collection of cones.

\begin{theorem}\label{maintheoremlipschitzc1}
Given an arbitrary subset $E$ of $\Rn$, a linear subspace $X\subset\R^n$, the orthogonal projection $P:=P_{X}:\R^n\to X$, and two functions $f:E \to \R, \: G: E \to \Rn$, the following is true. There exists a Lipschitz convex function $F:\R^n\to\R$ of class $C^1$ such that $F_{|_E}=f$, $(\nabla F)_{|_E} =G$, and $X_{F}=X$, if and only if the following conditions are satisfied.
\begin{itemize}
\item[$(i)$] $G$ is continuous and bounded and $f(x)\geq f(y)+\langle G(y), x-y\rangle$ for all $x, y\in E$.
\item[$(ii)$] $Y:=\textrm{span}\left(\{G(x)-G(y) : x, y\in E\}\right)\subseteq X$.
\item[$(iii)$] If $Y \neq X$ and we denote $k= \dim Y$ and $d= \dim X$, there exist points $p_1, \ldots, p_{d-k} \in \R^n \setminus \overline{E}$, a number $\varepsilon \in (0,1),$ and linearly independent normalized vectors $w_1, \ldots, w_{d-k} \in X \cap Y^{\perp}$ such that, for every $j=1, \ldots , d-k,$ the cone $V_j := \lbrace x\in \R^n \: : \:  \varepsilon \langle w_j,x- p_j \rangle \geq |P_Y(x-p_j)|\rbrace$ does not contain any point of $\overline{E}.$ Here $P_Y:\R^n\to Y$ denotes the orthogonal projection onto $Y$.
\item[$(iv)$] If $(x_k)_k, (z_k)_k$ are sequences in $E$ such that $(P_X(x_k))_k$ is bounded and
$$
\lim_{k\to\infty} \left( f(x_k)-f(z_k)- \langle G(z_k), x_k-z_k \rangle \right) = 0,
$$
then $\lim_{k\to\infty} | G(x_k)-G(z_k) | = 0$.
\end{itemize}
Moreover, there exists a constant $C(n)>0$ only depending on $n$ such that, whenever these conditions are satisfied, the extension $F$ can be taken so that
$$
\lip(F) = \sup_{x\in \R^n} | \nabla F(x)| \leq C(n) \sup_{y\in E} |G(y)|.
$$
\end{theorem}

\medskip

Finally, let us turn our attention to a geometrical problem which is closely related to our results.

\begin{problem}\label{interpolating problem for convex hypersurfaces}
Given an arbitrary subset $E$ of $\R^n$ and a unitary vector field $N:E\to\R^n$, what conditions will be necessary and sufficient in order to guarantee the existence of a convex hypersurface $M$ of class $C^1$ with the properties that $E\subset M$ and $N(x)$ is normal to $M$ at each $x\in E$?.
\end{problem}

Our solution to this problem is as follows. We say that a subset $W$ of $\R^n$ is a (possibly unbounded) convex body provided that $W$ is closed and convex, with nonempty interior. Assuming, as we may, that $0\in\textrm{int}(W)$, we will say that $W$ is of class $C^1$ provided that its Minkowski functional
$$
\mu_{W}(x)=\inf\{\lambda>0 : \tfrac{1}{\lambda}x\in W\}
$$
is of class $C^1$ on the open set $\R^n\setminus \mu_{W}^{-1}(0)$. This is equivalent to saying that $W$ can be locally parametrized as a graph $(x_1, \ldots, x_{n-1}, g(x_{1}, \ldots, x_{n-1}))$ (coordinates taken with respect to an appropriate permutation of the canonical basis of $\R^n$), where $g$ is of class $C^{1}$. We will denote
$$
n_{W}(x)=\frac{\nabla\mu_{W}(x)}{|\nabla \mu_{W}(x)|}, \quad x\in\partial W,
$$
the outer normal to $\partial W$.

\begin{theorem}\label{main geometric corollary}
Let $E$ be an arbitrary subset of $\R^n$, $N:E\to\mathbb{S}^{n-1}$ a continuous mapping, $X$ a linear subspace of $\R^n$, and $P:\R^n\to X$ the orthogonal projection. Then there exists a (possibly unbounded) convex body $W$ of class $C^1$ such that $E\subset \partial W$, $0\in\textrm{int}(W)$, $N(x)=n_{W}(x)$ for all $x\in E$, and $X=\textrm{span}\left( n_{W}(\partial W)\right)$, if and only if the following conditions are satisfied:
\begin{enumerate}
\item $\langle N(y), x-y\rangle\leq 0$ for all $x, y\in E$.
\item For all sequences $(x_k)_k$, $(z_k)_k$ contained in $E$ with $(P(x_k))_k$ bounded, we have that
$$
\lim_{k\to\infty}\langle N(z_k), x_k-z_k\rangle=0 \implies \lim_{k\to\infty}|N(z_k)-N(x_k)|=0.
$$
\item $0<\inf_{y\in E}\langle N(y), y\rangle$.
\item Denoting $d=\textrm{dim}(X)$, $Y=\textrm{span} (N(E))$, $\ell=\textrm{dim}(Y)$, we have that $Y\subseteq  X$, and if $Y\neq X$ and $P_Y:\R^n\to Y$ is the orthogonal projection then there exist linearly independent normalized vectors $w_1, \ldots, w_{d-k} \in X \cap Y^{\perp}$, points $p_1, \ldots, p_{d-\ell} \in \R^n$, and a number $\varepsilon\in (0, 1)$ such that
$$
\left( \overline{E}\cup \{0\}\right)\cap \left( \bigcup_{j=1}^{d-\ell} V_j\right)=\emptyset,
$$
where $V_{j}:=\lbrace x\in \R^n \: : \:  \varepsilon \langle w_j,x- p_j \rangle \geq |P_Y(x-p_j)|\rbrace$ for every $j=1, \ldots, d- \ell.$
\end{enumerate}
\end{theorem}
As before, in the case that $X=\textrm{span} (N(E))$, the above result is much easier to use.

\begin{corollary}\label{geometric corollary}
Let $E$ be an arbitrary subset of $\R^n$, $N:E\to\mathbb{S}^{n-1}$ a continuous mapping, $X$ a linear subspace of $\R^n$ such that $X=\textrm{span}(N(E))$, and $P:\R^n\to X$ the orthogonal projection. Then there exists a (possibly unbounded) convex body $W$ of class $C^1$ such that $E\subset \partial W$, $0\in\textrm{int}(W)$, $N(x)=n_{W}(x)$ for all $x\in E$, and $X=\textrm{span}\left( n_{W}(\partial W)\right)$, if and only if the following conditions are satisfied:
\begin{enumerate}
\item $\langle N(y), x-y\rangle\leq 0$ for all $x, y\in E$.
\item For all sequences $(x_k)_k$, $(z_k)_k$ contained in $E$ with $(P(x_k))_k$ bounded, we have that
$$
\lim_{k\to\infty}\langle N(z_k), x_k-z_k\rangle=0 \implies \lim_{k\to\infty}|N(z_k)-N(x_k)|=0.
$$
\item $0<\inf_{y\in E}\langle N(y), y\rangle$.
\end{enumerate}
\end{corollary}

\section{Proof of Theorem \ref{decomposition theorem}.}

Let us first recall some terminology from \cite{Azagra}.
We say that a function $C:\R^n\to\R$ is a $k$-dimensional {\em corner function} on $\R^n$ if it is of the
form
    $$
C(x)=\max\{\, \ell_1 +b_1, \, \ell_2 +b_2, \, \ldots, \, \ell_k +b_k
\, \},
    $$
where the $\ell_j:\R^n\to\R$ are linear functions such that the
functions $L_{j}:\R^{n+1}=\R^n\times \R\to\R$ defined by $L_{j}(x,
x_{n+1})=x_{n+1}-\ell_j(x)$, $1\leq j\leq k$, are linearly
independent in $(\R^{n+1})^*$, and the $b_j\in\R$. This is equivalent to saying that the functions $\{\ell_2-\ell_1, \ldots, \ell_k -\ell_1\}$ are linearly independent in $(\R^n)^*$. 

We also say that a convex function $f:\R^n\to\R$ is supported by $C$ at a point $x\in \R^n$
provided we have $C\leq f$ and $C(x)=f(x)$.

Now let us prove Theorem \ref{decomposition theorem}. 

{\bf Case 1.} We will first assume that $f$ is differentiable (and therefore of class $C^1$, since $f$ is convex). If $f$ is affine, say $f(x)=a\langle u, x\rangle +b$, then the result is trivially true with $X=\{0\}$, $c(0)=b$, and $v=a u$. On the other hand, if $f$ is essentially coercive then the result also holds obviously with $X=\R^n$, $v=0$, and $c=f$. So we may assume that $f$ is neither affine nor essentially coercive. In particular there exist $x_0, y_0\in\R^n$ with $Df(x_0)\neq Df(y_0)$. It is then clear
that $L_1(x, x_{n+1})=x_{n+1}-Df(x_0)(x)$ and $L_2(x,
x_{n+1})=x_{n+1}-Df(y_0)(x)$ are two linearly independent linear
functions on $(\R^{n+1})^*$, hence $f$ is supported at $x_0$ by the
two-dimensional corner $x\mapsto \max\{f(x_0)+Df(x_0)(x-x_0),
f(y_0)+Df(y_0)(x-y_0)\}$. 

Let us then define $k$ as the greatest
integer number so that $f$ is supported at $x_0$ by a
$(k+1)$-dimensional corner. By assumption we have $1\leq k< n$.
Then we also have that there exist $\ell_1, \ldots, \ell_{k+1}\in
(\R^n)^{*}$ with $L_{j}(x, x_{n+1})=x_{n+1}-\ell_j(x)$, $j=1, \ldots,
k+1$, linearly independent in $(\R^{n+1})^{*}$, and $b_{1}, \ldots,
b_{k+1}\in\R$, so that $C=\max_{1\leq j\leq k+1}\{\ell_j +b_j\}$
supports $f$ at $x_0$.

Observe that the $\{L_{j}-L_1\}_{j=2}^{k+1}$ are linearly
independent in $(\R^{n+1})^{*}$, hence so are the
$\{\ell_{j}-\ell_1\}_{j=2}^{k+1}$ in $(\R^n)^{*}$, and therefore
$\bigcap_{j=2}^{k+1}\textrm{Ker}\, (\ell_{j}-\ell_1)$ has
dimension $n-k$. Then we can find linearly independent vectors
$w_1, \ldots, w_{n-k}$ such that $\bigcap_{j=2}^{k+1}\textrm{Ker}\,
(\ell_{j}-\ell_1)=\textrm{span}\{w_1, \ldots, w_{n-k}\}$.

Now, given any $y\in \R^n$, if $\frac{d}{dt} (f-\ell_1)(y+t
w_q)|_{t=t_{0}}\neq 0$ for some $t_0$ then
$Df(y+t_0 w_q)-\ell_1$ is linearly independent with
$\{\ell_{j}-\ell_1\}_{j=2}^{k+1}$, which implies that $(x,
x_{n+1})\mapsto x_{n+1}-Df(y+t_0w_q)$ is linearly independent with
$L_1, \ldots, L_{k+1}$, and therefore the function
$$x\mapsto \max\{\ell_1(x)+ b_1, \ldots, \ell_{k+1}(x)+b_{k+1},
Df(y+t_0 w_q)(x-y-t_0 w_q)+f(y+t_0 w_q)\}$$ is a $(k+2)$-dimensional corner supporting $f$ at $x_0$, which
contradicts the choice of $k$. Thus we must have
\begin{equation}\label{f' is 0 in directions parallel to X perp}
\frac{d}{dt}(f-\ell_1)(y+tw_q)=0 \,\,\,
\textrm{ for all } \, y\in \R^n, t\in\R \, \textrm{ with } \, y+tw_q\in \R^n,  \, q=1, ..., n-k.
\end{equation}\label{f is constant in directions parallel to X perp}
This implies that
\begin{equation}
(f-\ell_1)(y+\sum_{j=1}^{n-k}t_{j}w_j)=(f-\ell_1)(y)
\end{equation}
if $y\in \R^n$ and $t_1,\ldots, t_{n-k} \in \R$.
Let $P$ be the orthogonal projection of $\R^n$ onto the subspace
$X:=\textrm{span}\{w_1, \ldots, w_{n-k}\}^{\bot}$. For each $z\in X $ we may define
$$
\widetilde{c}(z)=(f-\ell_1)(z+\sum_{j=1}^{n-k}t_j w_j)
$$
if $z+\sum_{j=1}^{n-k}t_j w_j \in \R^n$ for some $t_1, \ldots, t_{n-k}$. It is clear that $\widetilde{c}:X\to\R$ is well defined and
convex, and satisfies
$$
f-\ell_1=\widetilde{c}\circ P.
$$
Now let us write $$\ell_1(x)=\langle u, x \rangle + \langle v, x\rangle,$$ where $u\in X$ and $v\in X^{\perp}$. We then have
$$
f(x)=c(P(x))+\langle v,x\rangle,
$$
where $c:X\to \R$ is defined by $$c(x)=\widetilde{c}(x)+\langle u, x\rangle.$$

Moreover, since $\bigcap_{j=2}^{k+1}\textrm{Ker}\,
(\ell_{j}-\ell_1) =X^{\perp}$, it is clear that the restriction of the corner function $C=\max_{1\leq j\leq k+1}\{\ell_j +b_j\}$ to $X$ is a $(k+1)$-dimensional corner function on $X$, which has dimension $k$, and it is obvious that $(k+1)$-dimensional corner functions on $k$-dimensional spaces are essentially coercive; therefore, because $c(x)\geq C(x)$ for all $x\in X$, we deduce that $c$ is essentially coercive.

Now let us see that $X$ is the only linear subspace of $\R^n$ for which $f$ admits a decomposition of the form
\begin{equation}\label{X decomposition}
f(x)=c(P_{X}(x))+\langle v, x\rangle,
\end{equation}
with $c$ essentially coercive and $v\in X^{\perp}$. Assume that we have two subspaces $Z_1, Z_2$ for which \eqref{X decomposition} holds, say
\begin{equation}\label{Z1 decomposition}
f(x)=\varphi_1(P_{Z_1}(x))+\langle \xi_1, x\rangle,
\end{equation}
and
\begin{equation}\label{Z2 decomposition}
f(x)=\varphi_2(P_{Z_2}(x))+\langle \xi_2, x\rangle,
\end{equation}
with $\varphi_j$ essentially coercive and $\xi_j\in X_{j}^{\perp}$.
In order to show that $Z_1=Z_2$ it is enough to check that $Z_1^{\perp}=Z_{2}^{\perp}$. Suppose this equality does not hold; then, either $\in Z_{1}^{\perp}\setminus Z_{2}^{\perp}\neq\emptyset$ or $\in Z_{2}^{\perp}\setminus Z_{1}^{\perp}\neq\emptyset$. Assume for instance that there exists $\xi_0\in Z_{1}^{\perp}\setminus Z_{2}^{\perp}$. Then, on the one hand \eqref{Z1 decomposition} implies that the function $t\mapsto f(t\xi_0)=\varphi_1(0)+t\langle \xi_1, \xi_0\rangle$ is linear, and on the other hand \eqref{Z2 decomposition} implies that the same function $t\mapsto f(t\xi_0)=\varphi_2(P_{Z_2}(t\xi_0))+t\langle \xi_2, \xi_0\rangle$ is essentially coercive (indeed, we have $\lim_{|t|\to\infty}|P_{Z_2}(t\xi_0)|=\infty$ because $\xi_0\notin Z_2^{\perp}$). This is absurd, so we must have $Z_1^{\perp}\subset Z_{2}^{\perp}$. By a similar argument, just changing the roles of $Z_1$ and $Z_2$, we also obtain that $Z_2^{\perp}\subset Z_{1}^{\perp}$. Therefore $Z_1^{\perp}=Z_{2}^{\perp}$, as
  we
  wanted to check.
\footnote{It is worth noting that the preceding argument also shows that the dimension of $X_f$ is $k$, the largest integer such that $f$ is supported at some point by a $(k+1)$-dimensional corner function.
In particular, it follows that a function is essentially coercive in $\R^n$ if and only if it is supported by an $(n+1)$-dimensional corner function.}

Next, let us see that $\xi_1=\xi_2$. For every $v\in Z_{1}^{\perp}$ we have
$$\varphi_1(0)+\langle \xi_1, v\rangle=f(v)=\varphi_2(0)+\langle\xi_2, v\rangle.
$$
Since the equality of two affine function imply the equality of their linear parts, we have that
$$
\langle \xi_1, v\rangle=\langle\xi_2, v\rangle
$$
for all $v\in Z_{1}^{\perp}$, and because $\xi_1, \xi_2\in Z_1^{\perp}$ this shows that $\xi_1=\xi_2$.

Once we know that $X_1=X_2$ and $\xi_1=\xi_2$, it immediately follows from \eqref{Z1 decomposition} and \eqref{Z2 decomposition} that $\varphi_1=\varphi_2$.
This shows that the decomposition is unique.

Finally let us prove that if $f$ is essentially coercive in the direction of  a subspace $Y$ (say that there exists a linear form $\ell$ on $\R^n$ such that $|f(x)-\ell(x)|\to\infty$ as $|P_{Y}(x)|\to\infty$), then $Y\subseteq X_f$. Indeed, otherwise there would exist a vector $\xi\in X^{\perp}\setminus Y^{\perp}$, and the function
$$
\R\ni t\mapsto f(t\xi)=c(P_{X}(t\xi))+t\langle v, \xi\rangle= c(0)+t\langle v, \xi\rangle
$$ 
would be affine, hence so would be the function
$$
\R\ni t\mapsto f(t\xi)-\ell(t\xi).
$$
But this function cannot be affine, because $\xi\notin Y^{\perp}$ implies that $|P_{Y}(t\xi)|\to\infty$ as $|t|\to\infty$, and we have $|f(x)-\ell(x)|\to\infty$ as $|P_{Y}(x)|\to\infty$. This completes the proof of Theorem \ref{decomposition theorem} in the case that $f$ is everywhere differentiable.

{\bf Case 2.} In the case that $f:\R^n\to\R$ is convex but not everywhere differentiable, we can use \cite[Theorem 1.1]{Azagra} in order to find a $C^1$ (or even real analytic) convex function $g:\R^n\to\R$ such that $f-1\leq g\leq f$. Then we may apply Case 1 in order to find a unique subspace $X\subseteq \R^n$, an essentially coercive convex function $C:X\to\R$ and a vector $v\in X^{\perp}$ such that
$$
g(z)=c(P(z))+\langle v, z\rangle
$$
for all $z\in\R^n$. Now take $x\in X$ and $\xi\in X^{\perp}$. The function $\R\ni t\mapsto g(t\xi)$, is affine, and because $f\leq g+1$ and $f$ is convex, so must be the function $\R\ni t\mapsto f(t\xi)$, and with the same linear part (this immediately follows form the fact that the only convex functions which are bounded above on $\R$ are constants). This shows that
$$
f(x+t\xi)=f(x)+t\langle v, \xi \rangle
$$ 
for every $x\in X$, $\xi\in X^{\perp}$, $t\in\R$. Equivalently, we can write
$$
f(z)=\varphi(P(z))+\langle v, z\rangle \, \textrm{ for all } z\in\R^n,
$$
where $\varphi:X\to \R$ is defined by $\varphi(x)=f(x)$ for all $x\in X$. Moreover, $\varphi$ is essentially coercive because so is $g_{|_X}$ and we have $|f-g|\leq 1$.
This shows the existence of the decomposition in the statement. The uniqueness of the decomposition, as well as the last part of the statement of Theorem \ref{decomposition theorem}, follows by the same arguments as in Case 1, because that part of the proof does not use the differentiability of $f$.
The proof of Theorem \ref{decomposition theorem} is thus complete. \qed

\section{Necessity of Theorem \ref{MainTheorem1WithPThenP}} 

Let $F$ be a convex function of class $C^1(\R^n)$ such that $(F,\nabla F)$ extends $(f,G)$ from $E$, and $X_F=X.$ 

\subsection{Condition $(i)$} The inequality $f(x)-f(y)-\langle G(y), x-y \rangle \geq 0$ for all $x,y\in E$ follows from the fact that $F$ is convex and differentiable with $(F, \nabla F)=(f,G)$ on $E.$

\subsection{Condition $(ii)$} Assume that $\left( |\nabla F (x_k) | \right)_k$ tends to $+\infty$ for a sequence $(x_k)_k  \subset \Rn$ but
$$
\frac{\langle \nabla F (x_k), x_k \rangle - F(x_k)}{|\nabla F(x_k)|}
$$
does not go to $+\infty.$ Then, passing to a subsequence, we may assume that there exists $M>0$ such that $\langle \nabla F (x_k), x_k \rangle - F(x_k) \leq M | \nabla F (x_k) |$ for all $k.$ We denote $z_k = 2 M \frac{\nabla F (x_k)}{|\nabla F (x_k)|}.$ By convexity, we have, for all $k,$ that
$$
0 \leq F(z_k)-F(x_k)-\langle \nabla F(x_k), z_k-x_k \rangle \leq F(z_k)-M | \nabla F (x_k) |,
$$
which contradicts the assumption that $|\nabla F (x_k) | \to \infty.$

\subsection{Condition $(iii)$}
Making use of Theorem \ref{decomposition theorem} and bearing in mind that $X_F=X,$ we can write $F =c \circ P_X + \langle v, \cdot \rangle,$ where $P_X: \Rn \to X$ is the orthogonal projection onto the subspace $X,$ the function $c: X \to \R$ is convex and essentially coercive on $X$, and $v \perp X.$ It is easy to see that $c$ is differentiable on $X$ and that $\nabla F(x)= \nabla c (P_X(x)) + v$ for all $x\in \Rn.$ Since $F=G$ on $E,$ we easily get $G(x)-G(y) \in X$ for all $x,y\in E.$ 

\subsection{Condition $(v)$}
Let us consider sequences $(x_k)_k, (z_k)_k$ on $E$ such that $(P_X(x_k))_k$ and $(\nabla F(z_k))_k$ are bounded and
\begin{equation}\label{limitnecessity}
\lim_{k\to\infty} \left( F(x_k)-F(z_k)- \langle \nabla F(z_k), x_k-z_k \rangle \right)=0.
\end{equation}

Suppose that $ | \nabla F(x_k)-\nabla F(z_k) |$ does not converge to $0.$ Then, using that $(P_X(x_k))_k$ is bounded, there exist some $x_0 \in X$ and $\varepsilon>0$ for which, possibly after passing to a subsequence, $P_X(x_k)$ converges to $x_0$ and $ | \nabla F(x_k)-\nabla F(z_k) | \geq \varepsilon$ for every $k$. By using the decomposition $F= c \circ P_X + \langle v, \cdot \rangle$ and some elementary properties of orthogonal projections with \eqref{limitnecessity} we obtain
$$
\lim_{k\to\infty} \left( c ( P_X(x_k))-c (P_X( z_k))- \langle \nabla c (P_X(z_k)) , P_X(x_k)-P_X(z_k) \rangle \right) = 0. 
$$
Since $ \nabla F (y)-v = \nabla c (P_X( y))$ for all $y\in \Rn$ we have that $ (\nabla c (P_X(z_k)))_k$ is bounded and $$| \nabla c (P_X(x_k))-\nabla c (P_X(z_k)) | \geq \varepsilon$$ for every $k.$ Besides
$$
\lim_{k\to\infty} \left( c ( x_0)-c (P_X( z_k))- \langle \nabla c (P_X(z_k)) , x_0-P_X(z_k) \rangle  \right) =0.
$$
The contradiction follows from the following Lemma.

\begin{lemma}
Let $h: X \to \R$ be a differentiable convex function, $x_0\in X$, and $(y_k)_k$ be a sequence in $X$ such that $( \nabla h (y_k) )_k$ is bounded and
$$
 \lim_{k\to\infty}  \left( h(x_0)-h(y_k)-\langle \nabla h(y_k), x_0-y_k \rangle \right)= 0.
$$
Then $\lim_{k\to\infty} | \nabla h(x_0)- \nabla h(y_k) | = 0.$
\end{lemma}
\begin{proof}
Suppose not. Then, up to extracting a subsequence, we would have $| \nabla h(x_0)- \nabla h(y_k) |\geq \varepsilon,$ for some positive $\varepsilon$ and for every $k.$ Now, for every $k,$ we set
$$
\alpha_k:=  h(x_0)-h(y_k)-\langle \nabla h(y_k), x_0-y_k \rangle, \quad v_k:= \frac{\nabla h(y_k)- \nabla h(x_0) }{|\nabla h(y_k)- \nabla h(x_0) |}.
$$
In \cite[Lemma 2.1]{AzagraMudarra2017PLMS} it is proved that $\alpha_k=0$ implies $| \nabla h(x_0)- \nabla h(y_k) |=0,$ which is absurd. Thus we must have $\alpha_k >0$ for every $k.$ By convexity we have
\begin{align*}
 \sqrt{\alpha_k} \langle \nabla h (x_0 & + \sqrt{\alpha_k} v_k),  v_k \rangle \geq  h(x_0+\sqrt{\alpha_k} v_k ) - h(x_0) \\
 &  \geq h(y_k)+ \langle \nabla h(y_k), x_0+\sqrt{\alpha_k} v_k -y_k \rangle - h(x_0) \\
& = -\alpha_k + \sqrt{\alpha_k} \langle \nabla h(y_k),  v_k \rangle 
\end{align*}
for all $k$.
Hence, we obtain
$$
\langle \nabla h (x_0 + \sqrt{\alpha_k} v_k)- \nabla h(x_0),  v_k \rangle \geq - \sqrt{\alpha_k}+ |\nabla h(y_k)- \nabla h(x_0) | \geq - \sqrt{\alpha_k}+ \varepsilon.
$$
But the above inequality is impossible, as $\nabla h$ is continuous and $\alpha_k \to 0$. 
\end{proof}

\subsection{Condition $(iv)$}

By applying Theorem \ref{decomposition theorem} we may write
$$
F(x)=c(P_{X}(x))+\langle v, x\rangle,
$$
with $c:X\to\R$ convex and essentially coercive, and  $v\perp X$. This implies that 
$$
X=\textrm{span}\{\nabla c(x)-\nabla c(y) : x, y\in X\},
$$
and because $\nabla F=\nabla (c\circ P_X) +v$, also that
$$
X=\textrm{span}\{\nabla F(x)-\nabla F(y) : x, y\in \R^n\}.
$$
Let us denote $Y:=\textrm{span}\{\nabla F(x)-\nabla F(y) : x, y\in E\}\subset X$ and assume that $Y\neq X.$ Let $k$ and $d$ denote the dimensions of $Y$ and $X$ respectively. We can find points $x_0,x_1, \ldots, x_k \in E$ such that $Y =\textrm{span} \lbrace \nabla F(x_j)-\nabla F(x_0) \: : \: j=1, \ldots, k \rbrace.$ We claim that there exists $p_1 \in \R^n$ such that $\nabla F(p_1)-\nabla F(x_0) \notin Y.$ Indeed, otherwise we would have that $\nabla F(p)-\nabla F(x_0) \in Y$ for all $p\in \R^n,$ which implies that
$$
\nabla F(p)-\nabla F(q)=(\nabla F(p)-\nabla F(x_0)) - (\nabla F(q)-\nabla F(x_0)) \in Y, \quad \text{for all} \quad p,q \in \R^n.
$$
This is a contradiction since $X \neq Y.$ Then the subspace $Y_1$ spanned by $Y$ and the vector $\nabla F(p_1)-\nabla F(x_0) $ has dimension $k+1.$ If $d=k+1,$ we are done. If $d>k+1,$ using the same argument as above, we can find a point $p_2 \in \R^n$ such that $\nabla F(p_2)-\nabla F(x_0) \notin Y_1.$ By induction, we obtain points $p_1, \ldots ,p_{d-k} \in \R^n$ such that the set $ \lbrace \nabla F(p_j)- \nabla F(x_0) \rbrace_{j=1}^{d-k}$ is linearly independent and $X = Y \oplus \textrm{span}\lbrace \nabla F(p_j)- \nabla F(x_0) \: : \: j=1, \ldots ,d-k \rbrace,$ which shows that
$$
X=\textrm{span}\left\{u-w : u, w\in \nabla F(E)\cup\{\nabla F(p_1), \ldots, \nabla F(p_{d-k})\}\, \right\}.
$$
This shows the necessity of $(iv)(a)$. Obviously we have $\nabla F(p_j)-\nabla F(x_0)\in X\setminus Y$ for all $j=1, \ldots, d-k$, and we claim that
$$
p_{j}\in\R^n\setminus \overline{E} \quad \textrm{for all} \quad j= 1, \ldots, d-k.
$$
Indeed, if there exists a sequence $(q_\ell)_\ell \subset E$ with $(q_\ell)_\ell \to p_j$ for some $j=1,\ldots,d-k,$ then, because $Y$ is closed and $ \nabla F$ is continuous, $\nabla F(p_j)-\nabla F(x_0) = \lim_\ell \left( \nabla F(q_\ell)-\nabla F(x_0) \right) \in Y,$ which is a contradiction. By the (already shown) necessity of condition $(v)$, applied with $E^{*}=E\cup\{p_1, \ldots, p_{d-k}\}$ in place of $E$, we have that
\begin{equation}\label{differences of derivatives must go to 0}
\lim_{\ell\to\infty} | \nabla F(x_\ell)- \nabla F(z_\ell) | = 0
\end{equation}
whenever $(x_\ell)_\ell, (z_\ell)_\ell$ are sequences in $E^{*}$ such that $(P_{X} (x_\ell))_\ell$ and $(\nabla F(z_\ell))_\ell$ are bounded and
$$
\lim_{\ell\to\infty} \left( F(x_\ell)-F(z_\ell)- \langle \nabla F(z_\ell), x_\ell-z_\ell \rangle \right)= 0.
$$
But the fact that $\textrm{dist}(\nabla F(p_j)-\nabla F(x_0), Y)>0$ for each $j=1, \ldots, d-k$ prevents
the limiting condition \eqref{differences of derivatives must go to 0} from holding true with $(z_\ell)_\ell \subset\{p_1, \ldots, p_{d-k}\}$ and $(x_\ell)_\ell\subset E$. This implies that the inequalities
\begin{eqnarray*}
& & F(p_j)\geq F(p_i)+\langle\nabla F(p_i), p_j-p_i\rangle, \quad 1 \leq i,j \leq d-k, \: i \neq j,\\
& & F(p_j)\geq\sup_{z\in E, |\nabla F(z)|\leq N}\{F(z)+\langle \nabla F(z), p_j-z\rangle\},\quad  1 \leq j \leq d-k, \: N \in \N,   \: \textrm{ and }\\
& &  F(x) \geq F(p_j)+\langle \nabla F(p_j), x-p_j\rangle, \quad 1 \leq j \leq d-k, \: x\in \R^n,
\end{eqnarray*}
which generally hold by convexity of $F$, must all be strict. Moreover, the last of these inequalities, together with \eqref{differences of derivatives must go to 0}, also implies that
$$
\inf_{x\in E, \: |P_X(x)| \leq N}\lbrace F(x) - \max_{1 \leq j \leq d-k} \lbrace F(p_j) + \langle \nabla F(p_j), x-p_j\rangle  \rbrace  \rbrace >0
$$ 
for all $N \in \N$. Setting $w_j=\nabla F(p_j)$ and $\beta_j=F(p_j)$, $j=1, \ldots, d-k$, this shows the necessity of $(iv) (b)-(d)$.

\section{Sufficiency of Theorem \ref{MainTheorem1WithPThenP}}

First of all, with the notation of condition $(iv)$, if $Y\neq X$, we define $$E^{*}=E\cup\{p_1, \ldots , p_{d-k}\}$$ and extend the functions $f$ and $G$ to $E^{*}$ by setting
\begin{equation}\label{definitionnewdata}
f(p_j):=\beta_j, \quad G(p_j):= w_j \quad \text{for} \quad j=1, \ldots, d-k.
\end{equation}
If $Y=X$, we just set $E^{*}=E$ and ignore any reference to the points $p_j$ and their companions $w_j$ and $\beta_j$ in what follows.

\begin{lemma}\label{proofcompatiblenewdata}
We have:
\item[(a)]$X= \textrm{span} \left( \lbrace G(x)-G(y) \: : \: x,y \in E^* \rbrace \right).$  
\item[(b)] There exists $r>0$ such that $f(p_i)-f(p_j)-\langle G(p_j), p_i-p_j \rangle \geq r$ for all $1 \leq i \neq j \leq d-k.$
\item[(c)] For every $N \in \N,$ there exists $r_N>0$ with $f(p_i)-f(z)-\langle G(z), p_i-z \rangle \geq r_N$ for all $z\in E$ with $|G(z)| \leq N$ and all $1 \leq i \leq d-k.$ 
\item[(d)] For every $N \in \N,$ there exists $r_N>0$ with $f(x)-f(p_i)-\langle G(p_i), x-p_i \rangle \geq r_N$ for all $x\in E$ with $|P_X(x)| \leq N$ and all $1 \leq i \leq d-k.$ 
\end{lemma}
\begin{proof}
This follows immediately from condition $(iv)$ and the definitions of \eqref{definitionnewdata}.
\end{proof}

\begin{lemma}\label{conditioncw1forE*}
The jet $(f,G)$ defined on $E^{*}$ satisfies the inequalities of the assumption $(i)$ on $E^*$. Moreover,  if $(x_k)_k, (z_k)_k$ are sequences in $E^*$ such that $(P_X( x_k))_k$ and $(G(z_k))_k$ are bounded, then
$$
\lim_{k\to\infty} \left( f(x_k)-f(z_k)- \langle G(z_k), x_k-z_k \rangle \right)=0 \: \implies \lim_{k\to\infty} | G(x_k)-G(z_k) | = 0.
$$
\end{lemma}
\begin{proof}
Suppose that $(x_k)_k, (z_k)_k$ are sequences in $E^*$ such that $(P_X( x_k))_k$ and $(G(z_k))_k$ are bounded and $\lim_{k\to\infty} \left( f(x_k)-f(z_k)- \langle G(z_k), x_k-z_k \rangle \right)=0.$ In view of Lemma \ref{proofcompatiblenewdata} (b), (c) and (d), it is immediate that there exists $k_0$ such that either there is some $1\leq i \leq d-k$ with $x_k=z_k =p_i$ for all $k \geq k_0$ or else $x_k,z_k \in E$ for all $k \geq k_0.$ In the first case, the conclusion is trivial. In the second case, $\lim_{k\to\infty} |G(x_k)-G(z_k)|=0$ follows from condition $(v)$ of Theorem \ref{MainTheorem1WithPThenP}.
\end{proof}

We now consider the minimal convex extension of the jet $(f,G)$ from $E^*$, defined by
$$
m(x)=m(f,G, E^{*})(x) := \sup_{y\in E^*} \lbrace f(y)+\langle G(y),x-y \rangle \rbrace, \quad x\in \R^n.
$$

It is clear that $m,$ being the supremum of a family of affine functions, is a convex function on $\Rn.$ In fact, we have the following.

\begin{lemma}\label{finitenessminimal}
$m(x)$ is finite for every $x\in \Rn.$ In addition, $m=f$ on $E^*$ and $G(x) \in \partial m (x)$ for all $x\in E^*.$ 
\end{lemma}
Here $\partial m(x):=\{\xi\in\R^n : m(y)\geq m(x)+\langle \xi, y-x\rangle \textrm{ for all } y\in\R^n\}$ is the subdifferential of $f$ at $x$.
\begin{proof}
Fix a point $z_0\in E^*$. For any given point $x\in \Rn$ it is clear that there exists a sequence $(y_k)_k$ (possibly stationary) in $E^*$ such that 
$$
f(z_0)+ \langle G(z_0), x-z_0 \rangle \leq f(y_k)+ \langle G(y_k), x-y_k \rangle \quad \text{for all} \quad k, \quad  $$
and $f(y_k)+ \langle G(y_k), x-y_k \rangle \to m(x)$ as $ k \to \infty.$ On the other hand, by the first statement of Lemma \ref{conditioncw1forE*}, we have
\begin{align*}
f(y_k)+ \langle G(y_k), x-y_k \rangle \leq f(z_0)+ \langle G(y_k), x-z_0 \rangle.
\end{align*}
Then it is clear that $m(x)<+\infty$ when $ ( G(y_k) )_k$ is a bounded sequence. We next show that this sequence can never be unbounded. Indeed, in such case, by the condition $(ii)$ in Theorem \ref{MainTheorem1WithPThenP} (which obviously holds with $E^*$ in place of $E$), we would have a subsequence for which $\lim_{k\to\infty} |G(y_k)|= + \infty$ which in turn implies
$$
\lim_{k\to\infty} \frac{\langle G(y_k), y_k \rangle - f(y_k)}{|G(y_k)|} = +\infty.
$$
Hence, by the assumption on $( y_k )_k$ we would have
$$
\frac{f(y_k)-\langle G(y_k), y_k \rangle}{|G(y_k)|} \geq \frac{f(z_0)+\langle G(z_0), x-z_0 \rangle}{|G(y_k)|}- \Big \langle \frac{ G(y_k)}{|G(y_k)|}, x \Big \rangle.
$$
Since $\lim_{k\to\infty} |G(y_k)|= + \infty,$ the right-hand term is bounded below, and this leads to a contradiction. Therefore $m(x)< + \infty$ for all $x\in \Rn.$ In addition, by using the definition of $m$ and the first statement of Lemma \ref{conditioncw1forE*} for the jet $(f,G),$ we easily obtain that $m=f$ on $E^*$ and that $G(x)$ belongs to $\partial m(x)$ for all $x\in E^*$. 
\end{proof}

\begin{lemma}\label{subspaceminimal}
The function $m$ is essentially coercive in the direction of $X$, and in fact, with the notation of Theorem \ref{decomposition theorem} we have that
$$
X_{m}=X.
$$
\end{lemma}
\begin{proof}
By Lemma \ref{proofcompatiblenewdata} (a), we have $X= \textrm{span} \left( \lbrace G(x)-G(y) \: : \: x,y \in E^* \rbrace \right).$ Let us first see that $m$ is essentially coercive in the direction of $X$. If $X=\{0\}$ then $m$ is affine and the result is obvious. Therefore we can assume $\textrm{dim}(X)\geq 1$ and take points $x_0, x_1, \ldots, x_k\in E$ such that $\{v_1, \ldots, v_k\}$ is a basis of $X$, where
$$
v_j=G(x_j)-G(x_0), \,\,\, j= 1, \ldots, k.
$$
Then $$C(x)=\max\{ f(x_0)+\langle G(x_0), x-x_0\rangle, \, f(x_1)+\langle G(x_1), x-x_1\rangle, \ldots, \, f(x_k)+\langle G(x_k), x-x_k\rangle\}
$$
defines a $k$-dimensional corner function such that $$C(x)\leq m(x) \textrm{ for all } x\in\R^n,$$ and it is not difficult to see that $C$ is essentially coercive in the direction of $X$, hence so is $m$.

In particular, by Theorem \ref{decomposition theorem}, it follows that $X\subseteq X_m$.
 
Now, if $X_m\neq X$, we can take a vector $w\in X_m\setminus\{0\}$ such that $w\perp X$, and then we obtain, for all $t\in\R$, that
\begin{eqnarray*}
& & m(x_0+tw)-f(x_0)-\langle G(x_0), tw\rangle=\\
& & \sup_{z\in E}\{ f(z)-f(x_0)+\langle G(z)-G(x_0), tw\rangle +\langle G(z), x_0-z\}=\\
& & \sup_{z\in E}\{ f(z)-f(x_0) +\langle G(z), x_0-z\}\leq 0.
\end{eqnarray*}
By convexity, this implies that
$$
m(x_0+tw)=f(x_0)+\langle G(x_0), tw\rangle
$$
for all $t\in\R$, and in particular the function $\R\ni t\mapsto m(x_0+tw)$ cannot be essentially coercive, contradicting the assumption that $w\in X_m$. Therefore we must have $X_m=X$.
\end{proof}

Making use of Theorem \ref{decomposition theorem} in combination with Lemma \ref{subspaceminimal}, we can write 
\begin{equation}\label{formuladecompositionminimal}
m= c \circ P_X + \langle v, \cdot \rangle \quad \text{on} \quad \R^n,
\end{equation}
where $c: X \to \R$ is convex and essentially coercive on $X$ and $v \perp X.$ In addition, the subdifferential mappings of $m$ and $c$ satisfy the following.

\begin{claim}
Given $x\in \R^n$ and $\eta \in \partial m(x)$, we have $ \eta-v\in X$ and $\eta - v \in \partial c ( P_X(x)).$
\end{claim}
\begin{proof}
Suppose that $x \in \R^n$ and $\eta \in \partial m(x)$ but $\eta-v \notin X.$ Then we can find $w \in X^\perp$ with $\langle \eta -v , w \rangle =1.$ Using \eqref{formuladecompositionminimal} we get that
$$
\langle \eta , w \rangle \leq m(x+w)-m(x) = c(P_X(x+w))+\langle v, x+w \rangle - c(P_X(x))-\langle v,x \rangle = \langle v,w \rangle.
$$
This implies that $\langle \eta -v , w \rangle \leq 0$, a contradiction. This shows that $\eta -v \in X.$ Now, let $z\in X$ and $x\in \R^n.$ We have
\begin{align*}
c(z)-c(P_X(x))= m(z)-\langle v,z \rangle -m(x)+\langle v,x \rangle \geq  \langle \eta -v,z-x \rangle = \langle \eta -v,z-P_X(x) \rangle.
\end{align*}
Therefore, $\eta -v \in \partial c (P_X(x)).$ 
\end{proof}

By combining the previous Claim with the second part of Lemma \ref{finitenessminimal} we obtain that
\begin{equation}\label{derivativesc*inX}
G(x)-v \in \partial c (P_X(x))\subset X \quad \text{for all} \quad x\in E^*.
\end{equation}

\begin{lemma} \label{differentiabiltyclosurec^*}
The function $c$ is differentiable on $\overline{P_X(E^*)}$, and, if $y\in P_X(E^*)$, then $ \nabla c (y)= G(x)-v$, where $x\in E^*$ is such that $P_X( x) = y_.$
\end{lemma} 

\begin{proof}
Let us suppose that $c$ is not differentiable at some $y_0 \in \overline{P_X(E^*)}.$ Then, by the convexity of $c$ on $X,$ we may assume that there exist a sequence $(h_k)_k \subset X$ with $|h_k| \searrow 0$ and a number $\varepsilon >0$ such that 
$$
\varepsilon \leq \frac{ c(y_0+h_k)+c(y_0-h_k)-2c(y_0)}{|h_k|} \quad \text{for all} \quad k.
$$
We now consider sequences $(y_k)_k \subset P_X(E^*)$ and $ (x_k)_k \subset E^*$ with
$$
P_X (x_k) = y_k  \quad \text{and} \quad y_k \to y_0.
$$ 
In particular, the sequence $(P_X(x_k))_k$ is bounded. Since each $h_k$ belongs to $X,$ we can use \eqref{formuladecompositionminimal} to rewrite the last inequality as
\begin{equation}\label{inequalitydifferentiabilityc^*}
\varepsilon \leq \frac{ m(y_0+h_k)+m(y_0-h_k)-2m(y_0)}{|h_k|} \quad \text{for all} \quad k.
\end{equation}
 
By the definition of $m$ we can pick two sequences $(z_k)_k, ( \widetilde{z_k})_k \subset E^*$ with the following properties:
\begin{align*}
m(y_0+h_k) \geq f(z_k)+ \langle G(z_k), y_0 + h_k-z_k \rangle \geq m(y_0 + h_k)- \frac{|h_k|}{2^k}, \\
m(y_0- h_k) \geq f(\widetilde{z_k})+ \langle G(\widetilde{z_k}), y_0 - h_k-\widetilde{z_k} \rangle \geq m(y_0 - h_k)- \frac{|h_k|}{2^k}
\end{align*}
for every $k.$ We claim that $( G(z_k))_k$ must be bounded. Indeed, otherwise, possibly after passing to a subsequence and using the condition $(ii)$ of Theorem \ref{MainTheorem1WithPThenP}, we would obtain that
$$
\lim_{k\to\infty} |G(z_k)|= \lim_{k\to\infty} \frac{\langle G(z_k), z_k \rangle - f(z_k)}{|G(z_k)|} = +\infty.
$$
Due to the choice of $(z_k)_k$ we must have
\begin{align*}
m(y_0)& = \lim_{k\to\infty} (f(z_k)+ \langle G(z_k), x_0+h_k-z_k\rangle)\\
&= \lim_{k\to\infty} |G(z_k)| \left( \frac{f(z_k)-\langle G(z_k), z_k \rangle}{|G(z_k)|} + \Big \langle \frac{G(z_k)}{|G(z_k)|} , x_0+ h_k \Big \rangle \right) = - \infty,
\end{align*}
which is absurd. Similarly one can show that $(G(\tilde{z_k}))_k$ is bounded. Now we write
\begin{align*}
f(x_k) - & f(z_k) - \langle G(z_k), x_k -z_k \rangle \\
& = f(x_k)- \langle v, x_k \rangle - \left( m(y_0+k_k)-\langle v,y_0+h_k \rangle \right)\\
& \quad + m(y_0+h_k)-f(z_k)-\langle G(z_k), y_0+h_k-z_k \rangle \\
& \quad + \langle G(z_k)-v, y_0 +h_k-x_k \rangle. 
\end{align*}
By \eqref{formuladecompositionminimal}, the first term in the sum equals $c(P_X( x_k))- c ( y_0 +h_k ),$ which converges to $0$ because $P_X (x_k) \to  y_0$ and $c$ is continuous. Thanks to the choice of the sequence $(z_k)_k,$ the second term also converges to $0.$ From \eqref{derivativesc*inX}, we have $G(z_k)-v \in X$ for all $k,$ and then the third term in the sum is actually $\langle G(z_k)-v, y_0-P_X(x_k) + h_k \rangle,$ which converges to $0$, as $(G(z_k))_k$ is bounded and $P_X (x_k) \to y_0.$ We then have
$$
\lim_{k\to\infty} \left( f(x_k) - f(z_k) - \langle G(z_k), x_k -z_k \rangle \right)=0,
$$
where $(P_X (x_k))_k$ and $(G(z_k))_k$ are bounded sequences. We obtain from Lemma \ref{conditioncw1forE*} that $\lim_{k\to\infty} | G(x_k)-G(z_k)| = 0$, and similarly one can show that $\lim_{k\to\infty} | G(x_k)-G(\widetilde{z_k})| = 0.$ This obviously implies 
\begin{equation} \label{limitdifferencegradients}
\lim_{k\to\infty} | G(z_k)-G( \widetilde{z_k}) | =0.
\end{equation}
By the choice of the sequence $(z_k)_k, \: ( \widetilde{z_k} )_k$ and by inequality \eqref{inequalitydifferentiabilityc^*} we have, for every $k,$
\begin{align*}
\varepsilon & \leq  \frac{f(z_k)+ \langle G(z_k), y_0 + h_k-z_k \rangle}{|h_k|} + \frac{f(\widetilde{z_k})+ \langle G(\widetilde{z_k}), y_0 - h_k-\widetilde{z_k} \rangle}{|h_k|} \\
& \quad - \frac{f(z_k)+ \langle G(z_k), y_0-z_k \rangle + f(\widetilde{z_k})+ \langle G(\widetilde{z_k}), y_0 -\widetilde{z_k} \rangle}{|h_k|} \\
& = \Big \langle G(z_k)-G(\widetilde{z_k}), \frac{h_k}{|h_k|} \Big \rangle + \frac{1}{2^{k-1}} \leq | G(z_k)-G(\widetilde{z_k}) |  + \frac{1}{2^{k-1}} .
\end{align*}
Then \eqref{limitdifferencegradients} leads us to a contradiction. We conclude that $c$ is differentiable on $ \overline{P_X(E^*)}$. 

We now prove the second part of the Lemma. Consider $y\in P_X(E^*)$ and $x \in E^*$ with $P_X(x)=y.$ Using \eqref{derivativesc*inX}, we have $G(x)-v \in \partial c( y).$ Because $c$ is differentiable at $y,$ we further obtain that $G(x)-v = \nabla c(y)$. 
\end{proof}

In order to complete the proof of Theorem \ref{MainTheorem1WithPThenP}, we will need the following Lemma.

\begin{lemma}\label{extensionconvexcoercive}
Let $h:X \to \R$ be a convex and coercive function such that $h$ is differentiable on a closed subset $A$ of $X.$ There exists $H \in C^1(X)$ convex and coercive such that $H=h$ and $\nabla H = \nabla h$ on $A.$ 
\end{lemma}
\begin{proof}
Since $h$ is convex, its gradient $ \nabla h$ is continuous on $A$ (see \cite[Corollary 24.5.1]{Rockafellar} for instance). Then, for all $x,y\in A,$ we have
$$
0 \leq \frac{h(x)-h(y)-\langle \nabla h(y), x-y \rangle}{|x-y|} \leq \Big \langle \nabla h(x)-\nabla h(y), \frac{x-y}{|x-y|} \Big \rangle \leq |\nabla h(x)-\nabla h(y)| ,
$$
where the last term tends to $0$ as $|x-y| \to 0$ uniformly on $x,y\in K$ for every compact subset $K$ of $A.$ This shows that the pair $( h, \nabla h)$ defined on $A$ satisfies the conditions of the classical Whitney Extension Theorem for $C^1$ functions. Therefore, there exists a function $\widetilde{h} \in C^1(X)$ such that $\widetilde{h}= h$ and $\nabla \widetilde{h} = \nabla h$ on $A.$ We now define

\begin{equation}\label{definitionh}
\phi(x):=|h(x)-\widetilde{h}(x)| + 2 d(x,A)^2, \quad x\in X.
\end{equation}

\begin{claim}\label{claimdifferentiabilityphi}
$\phi$ is differentiable on $A$, with $\nabla \phi(x_0)=0$ for every $x_0\in A$.
\end{claim}
\begin{proof}
The function $d(\cdot, A)^{2}$ is obviously differentiable, with a null gradient, at $x_0$, hence we only have to see that $|h-\widetilde{h}|$ is differentiable, with a null  gradient, at $x_0$. Since $\nabla \widetilde{h}(x_0)=\nabla h(x_0),$ the Claim boils down to the following easy exercise: if two functions $h_1, h_2$ are differentiable at $x_0$, with $\nabla h_1(x_0)=\nabla h_2 (x_0)$, then $|h_1-h_2|$ is differentiable, with a null gradient, at $x_0$.
\end{proof}

Now, because $d(\cdot, A)^2$ is continuous and positive on $X \setminus A$, according to Whitney's approximation theorem \cite{Whitney} we can find a function $\varphi \in  C^{\infty}(X \setminus A)$ such that
\begin{equation}\label{approximationc}
|\varphi(x)-\phi(x)|\leq d(x, A)^2 \,\,\, \textrm{ for every } x\in X\setminus A,
\end{equation}

Let us define $\widetilde{\varphi}:X\to\R$ by $\widetilde{\varphi}=\varphi$ on $X\setminus A$ and $\widetilde{\varphi}=0$ on $A$.
\begin{claim}\label{claimdifferentiabilityvarphi}
The function $\widetilde{\varphi}$ is differentiable on $X$ and $\nabla \widetilde{\varphi} = 0$ on $A$.
\end{claim}
\begin{proof}
It is obvious that $\widetilde{\varphi}$ is differentiable on $\textrm{int}(A)\cup \left(X\setminus A\right)$ and $\nabla \widetilde{\varphi} = 0$ on $\textrm{int}(A)$. We only have to check that $\widetilde{\varphi}$ is differentiable on $\partial A$. If $x_0\in \partial A$ we have
$$
\frac{|\widetilde{\varphi}(x)-\widetilde{\varphi}(x_0)|}{|x-x_0|}=
\frac{|\widetilde{\varphi}(x)|}{|x-x_0|}\leq \frac{|\phi(x)|+d(x, A)^2}{|x-x_0|}\to 0
$$
as $|x-x_0|\to 0^{+}$, because both $\phi$ and $d(\cdot, A)^2$ vanish at $x_0$ and are differentiable, with null gradients, at $x_0$. Therefore $\widetilde{\varphi}$ is differentiable at $x_0$, with $\nabla \widetilde{\varphi}(x_0)=0$.
\end{proof}

Now we set $$g:=\tilde{h}+ \tilde{\varphi}$$ on $X$. It is clear that $g=h$ on $A.$ Also, by Claim \ref{claimdifferentiabilityvarphi}, $g$ is differentiable on $X$ with $\nabla g = \nabla h$ on $A.$ By combining \eqref{definitionh} and \eqref{approximationc} we easily obtain that 
$$
g(x) \geq \tilde{h}(x) + \phi(x) - d(x,A)^2 \geq h(x) \quad x\in X \setminus A.
$$
Therefore $g \geq h$ on $X$ and in particular $g$ is coercive on $X$, because so is $h$, by assumption.
 
We next consider the \textit{convex envelope} of $g.$ Recall that, for a function $\psi:X \to\R,$ the convex envelope of $\psi$ is defined by
$$
\textrm{conv}(\psi)(x)=\sup\{ \Phi(x) \, : \, \Phi \textrm{ is convex }, \Phi\leq \psi\}
$$
(another expression for  $\textrm{conv}(\psi)$, which follows from Carath\'eodory's Theorem, is
$$
\textrm{conv}(\psi)(x)=\inf\left\lbrace \sum_{j=1}^{n+1}\lambda_{j} \psi(x_j) \, : \, \lambda_j\geq 0,
\sum_{j=1}^{n+1}\lambda_j =1, x=\sum_{j=1}^{n+1}\lambda_j x_j \right\rbrace;
$$
see \cite[Corollary 17.1.5]{Rockafellar} for instance). The following result is a restatement of a particular case of the main theorem in \cite{KirchheimKristensen}; see also \cite{GriewankRabier}.

\begin{theorem}[Kirchheim-Kristensen]
If $\psi:X \to\R$ is differentiable and $\lim_{|x|\to\infty}\psi(x)=\infty$, then $\textrm{conv}(\psi)\in C^1(X)$.
\end{theorem}

If we define $$H=\textrm{conv}(g)$$ we immediately get that $H$ is convex on $X$ and $H\in C^{1}(X).$ By definition of $H$ we have that $h \leq H \leq g$ on $X,$ which implies that $H$ is coercive. Also, because $g=h$ on $A,$ we have that $H=h$ on $A.$ In order to show that $\nabla H = \nabla h$ on $A,$ we use the following well known criterion for differentiability of convex functions, whose proof is straightforward.
\begin{lemma}\label{differentiability criterion for convex functions}
If $\psi$ is convex, $\Phi$ is differentiable at $x$, $\psi\leq \Phi$, and $\psi(x)=\Phi(x)$, then $\psi$ is differentiable at $x$, with $\nabla\psi(x)=\nabla\Phi(x)$.
\end{lemma}
\noindent (This fact can also be phrased as: a convex function $\psi$ is differentiable at $x$ if and only if $\psi$ is superdifferentiable at $x$.)

Since $h$ is convex and $H$ is differentiable on $X$ with $h=H$ on $A$ and $h\leq H$ on $X,$ the preceding Lemma shows that $\nabla H = \nabla h$ on $A$. 

This completes the proof of Lemma \ref{extensionconvexcoercive}.
\end{proof}

Now we are able to finish the proof of Theorem \ref{MainTheorem1WithPThenP}. Setting $A: = \overline{P_X(E^*)},$ we see from Lemma \ref{differentiabiltyclosurec^*} that $c$ is differentiable on $A.$ Moreover, since $c :X \to \R$ is convex and essentially coercive on $X,$ there exists $\eta \in X$ such that $h:=c- \langle \eta, \cdot \rangle$ is convex, differentiable on $A$ and coercive on $X.$ Applying Lemma \ref{extensionconvexcoercive} to $h,$ we obtain $H \in C^1(X)$ convex and coercive on $X$ with $(H, \nabla H)=(h, \nabla h)$ on $A.$ Thus, the function $\varphi := H+ \langle \eta, \cdot \rangle$ is convex, essentially coercive on $X$ and of class $C^1(X)$ with $(\varphi, \nabla \varphi)=(c, \nabla c)$ on $A.$ We next show that $F:= \varphi \circ P_X + \langle v, \cdot \rangle$ is the desired extension of $(f,G).$ Since $\varphi$ is $C^1(X)$ and convex, it is clear that $F$ is $C^1(\R^n)$ and convex as well. Bearing in mind Theorem \ref{decomposition theorem} and the fact that $
 \varphi$ is essentially coercive, it follows that $X_F=X.$ Also, since $\varphi(y)=c(y)$ for $y\in P_X(E),$ we obtain from \eqref{formuladecompositionminimal} and Lemma \ref{finitenessminimal} that
$$
F(x)= \varphi(P_X(x))+\langle v, x \rangle = c (P_X(x))+\langle v, x \rangle=m(x)=f(x).
$$
Finally, from the second part of Lemma \ref{differentiabiltyclosurec^*}, we have, for all $x\in E$, that
$$
\nabla F(x)= \nabla \varphi (P_X(x)) + v = G(x)-v + v = G(x).
$$
The proof of Theorem \ref{MainTheorem1WithPThenP} is complete. \qed

\section{Necessity of Theorem \ref{maintheoremlipschitzc1}}

We already know that conditions $(i), (ii)$ and $(iv)$ are necessary for the existence of a convex function $F\in C^1(\R^n)$ with $(f,G)=(F, \nabla F)$ on $E$ and $X_F=X.$ Let us assume that $F$ is also Lipschitz, and let us prove that in this case condition $(iii)$ is satisfied as well. If $\textrm{Lip}(F)=0$ then $F$ is constant, so we have $X=X_{F}=\{0\}=Y$, and condition $(iii)$ is trivially satisfied. Otherwise we have $X=X_{F}\neq\{0\}$, and assuming that $Y\neq X$ we may find points $x_0, x_1, \ldots, x_k \in E$ and $p_1, \ldots, p_{d-k} \in \R^n \setminus \overline{E}$ such that 
$$
Y= \textrm{span} \lbrace G(x_j)-G(x_0) \: : \: j=1, \ldots, k \rbrace,
$$
$$
\nabla F(p_j)-G(x_0) \in X \setminus Y \quad \text{for every} \quad j=1, \ldots, d-k
$$
and the set $\lbrace \nabla  F(p_j)-G(x_0) \: : \: j=1, \ldots, d-k \rbrace$ is linearly independent. Now we define, for each $j=1,\ldots, d-k,$ the subspace $Y_j$ spanned by $Y$ and the vector $\nabla F(p_j)-G(x_0).$ Obviously we can find $w_j \in Y_j \cap Y^\perp$ with $|w_j|=1$ and $Y_j = Y \oplus [w_j],$ for every $j=1, \ldots, d-k.$ Moreover, $w_j$ can be taken so that 
$$
\mu_j : = \langle  \nabla F(p_j)-G(x_0), w_j \rangle >0, \quad \text{for all} \quad j=1, \ldots, d-k.
$$
Let us take $\varepsilon >0$ small enough so that
$$
\varepsilon < \frac{\mu_j}{2\lip(F)+2 \| G\|_\infty} \quad \text{for all} \quad j=1, \ldots, d-k.
$$
Note that, because $\mu_j \leq 2 \lip(F)$ for each $j,$ we have that $\varepsilon \leq 1.$ Now, assume that there exists some $x\in \overline{E}$ with $x\in V_j:= \lbrace x\in \R^n \: : \:  \varepsilon \langle w_j,x- p_j \rangle \geq |P_Y(x-p_j)|\rbrace$ for some $j=1,\ldots, d-k$ Using the convexity of $F$ we can easily write
\begin{align*}
F(x)& -F(p_j)-\langle \nabla F(p_j), x-p_j \rangle \\
& \leq \langle \nabla F(x)-\nabla F(p_j), x-p_j \rangle = \langle \nabla F(x)-G(x_0), x-p_j \rangle + \langle  G(x_0)-\nabla F(p_j), x-p_j \rangle \\
& = \langle \nabla F(x)-G(x_0), x-p_j \rangle - \mu_j \langle w_j,x-p_j \rangle + \langle P_Y ( G(x_0)-\nabla F(p_j) ), x-p_j \rangle.
\end{align*}
Since we are assuming that $x\in \overline{E},$ the continuity of $\nabla F$ yields $\nabla F(x)-G(x_0) \in Y.$ Then, the last term coincides with
\begin{align*}
\langle \nabla F(x) & -G(x_0), P_Y(x-p_j) \rangle - \mu_j \langle w_j,x-p_j \rangle + \langle P_Y ( G(x_0)-\nabla F(p_j) ), P_Y(x-p_j) \rangle \\
& \leq \left(  2 \| G \|_\infty + 2 \lip(F) \right) |P_Y(x-p_j)|- \mu_j \langle w_j,x-p_j \rangle \leq 0,
\end{align*}
where the last inequality follows from the definition of $\varepsilon$ and the fact that $x \in V_j.$ We have thus shown that
$$
F(x)-F(p_j)-\langle \nabla F(p_j), x-p_j \rangle =0,
$$ which implies, by condition $(CW^1),$ that $\nabla F(p_j)=\nabla F(x),$ where $x\in \overline{E}.$ It follows that $\nabla F(p_j) -G(x_0) = \nabla F(x)-G(x_0) \in Y,$ which contradicts the choice of $p_j.$ Therefore $\overline{E}$ and $\bigcup_{j=1}^{d-k} V_j$ are disjoint.

\section{Sufficiency of Theorem \ref{maintheoremlipschitzc1}. Keeping control of the Lipschitz constant}

If $m$ denotes the \textit{minimal convex extension} of the jet $(f,G)$ from $E,$  we can write
$$
m = c \circ P_{X_m} + \langle v, \cdot \rangle,
$$
where $v\in \R^n$ and $c:X_m \to \R$ is a coercive convex function. Moreover, we know that $X_m=Y= \textrm{span} \lbrace G(x)-G(y) \: : \: x,y \in E \rbrace
$ and therefore
\begin{equation}\label{firstdecompositionminimal}
m = c \circ P_{Y} + \langle v, \cdot \rangle.
\end{equation}

Let us prove some properties of $m,c$ and $v.$ 

\begin{lemma}\label{lemmalipschitzconstantdecomposition}
Let us denote by $K= \| G \|_\infty= \sup_{y\in E} |G(y)|$. We have that:
\begin{enumerate} 
\item[$(1)$] The function $m$ is $K$-Lipschitz on $\R^n.$
\item[$(2)$] The vector $v$ belongs to the subdifferential of $m$ at some point $y_0 \in Y$, and $|v| \leq K.$
\item[$(3)$] There exists points $x_1, \ldots, x_k \in E$ such that $\lbrace G(x_j)-v \rbrace_{j=1}^k$ is a basis of $Y.$
\item[$(4)$] The function $c$ is $2K$-Lipschitz on $Y.$
\item[$(5)$] There exists numbers $0 < \alpha \leq 2K$ and $\beta \in \R$ such that $c(y) \geq \alpha | y| + \beta$ for every $y\in Y$.
\end{enumerate}
\end{lemma}
\begin{proof}
\item[$(1)$] The function $m$ is a supremum of $K$-Lipschitz affine functions on $\R^n$ and therefore $m$ is $K$-Lipschitz as well.
\item[$(2)$] Since $c$ is coercive on $Y,$ there exists a point $y_0 \in Y$ with $c(y) \geq c(y_0)$ for every $y\in Y.$ We then have, for every $x\in \R^n$ that
$$
m(x) = c(P_Y(x)) + \langle v,x \rangle \geq c(y_0)+ \langle v,x \rangle = c(y_0)+ \langle v,y_0 \rangle + \langle v, x-y_0 \rangle = m(y_0) + \langle v, x-y_0 \rangle,
$$
which implies that $v\in \partial m (y_0).$ Since $m$ is $K$-Lipschitz, we obtain, for every $x\in \R^n,$
$$
K|x-y_0| + m(y_0) \geq m(x) \geq m(y_0) + \langle v, x-y_0 \rangle,
$$
which implies that $\langle v, \frac{x-y_0}{|x-y_0|} \rangle \leq K,$ for every $x \in \R^n \setminus \lbrace y_0 \rbrace.$ This shows that $|v| \leq K.$

\item[$(3)$] Recall that $\eta - v\in Y$ for every $\eta \in \partial m(x)$.  In particular we have $G(x)-v\in Y$ for every $x\in E.$ Let us take some $x_1 \in Y$ with $G(x_1)-v \neq 0.$ If $\dim(Y)=1,$ there is nothing to say. If $\dim(Y)>1,$ we claim that there exists some $x_2 \in E$ such that $G(x_2)-v$ and $G(x_1)-v$ are linearly independent. Indeed, assume that $G(x)-v$ and $G(x_1)-v$ are proportional for every $x\in E$. Then we would have for every $x,y \in E$ that
$$
G(x)-G(y) = (G(x)-v) + (v-G(y))
$$
is proportional to $G(x_1)-v$, hence $\dim(Y)=1$, a contradiction. Using an inductive argument we easily obtain $(3)$.
\item[$(4)$] follows at once from $(1), (2)$, and the fact that $c = m- \langle v, \cdot \rangle$ on $Y$.
\item[$(5)$] It is well known and easy to show that for every coercive convex function $c$  there exist numbers $\alpha >0$ and $\beta \in \R$ such that $c(y) \geq \alpha | y| + \beta$ for every $y\in Y.$ Now, because $c$ is $2K$-Lipschitz, we have that
$$
c(0)+ 2K |y| \geq c(y) \geq \alpha |y| + \beta, \quad y\in Y.
$$
This clearly implies that $\alpha \leq 2K.$
\end{proof}

\subsection{Defining new data}

Let us consider $w_1,\ldots, w_{d-k} \in Y ^{\perp} \cap X, \: \varepsilon \in (0,1), \:p_1, \ldots, p_{d-k}$ and $V_1,\ldots, V_{d-k}$ as in condition $(iii)$ of Theorem \ref{maintheoremlipschitzc1}. Using Lemma \ref{lemmalipschitzconstantdecomposition}(5), we consider a positive $T> 0$ large enough so that
$$
 (\varepsilon \: \alpha) T \geq 2 - \beta- \max_{j=1,\ldots, d-k} \lbrace \alpha |P_Y(p_j)| + m(p_j)-\langle v, p_j \rangle \rbrace
$$
and
$$
(\varepsilon \: \alpha) T \min_{1 \leq i \neq j \leq d-k} \lbrace 1- \langle w_i,w_j\rangle \rbrace \geq 1 + \max_{1 \leq i,j \leq d-k} \lbrace c(P_Y(p_j))-c(P_Y(p_i)) + \varepsilon \alpha \langle w_j, p_i-p_j \rangle \rbrace.
$$
Note that, since the vectors $\lbrace w_i \rbrace_{i=1}^{d-k}$ have norm equal to $1,$ then $\langle w_i, w_j \rangle =1$ if and only if $w_i=w_j,$ which is equivalent (as the vectors $\{w_1, \ldots , w_{d-k}\}$ are linearly independent) to $i=j.$ So it is clear that we can find a positive $T>0$ satisfying both inequalities. We define the following new data:
\begin{equation}\label{definitionnewjets}
q_j= p_j + T w_j, \quad f(q_j) = m(q_j)+1, \quad G(q_j)= v + \varepsilon \alpha w_j, \quad j=1, \ldots, d-k .
\end{equation}
Note that $q_i=q_j$ if and only if $p_i-p_j= T (w_j-w_i).$ Since $w_i\neq w_j$ whenever $i \neq j,$ it is clear that we can take $T$ large enough so that the points $q_i$ and $q_j$ are distinct if $i \neq j.$ On the other hand, because each $w_j$ is orthogonal to $Y,$ we immediately see that $q_j \in V_j$ and, in particular, $q_j \notin \overline{E}$ for every $j=1,\ldots, d-k.$ 

\begin{lemma} \label{lemmadatasarecompatible}
The following inequalities are satisfied.
\item[$(1)$] $f(q_j)-f(x)-\langle G(x), q_j-x \rangle \geq 1$ for every $x\in E,\: j=1,\ldots, d-k.$
\item[$(2)$] $f(x)-f(q_j)-\langle G(q_j), x-q_j \rangle \geq 1$ for every $x\in E,\: j=1,\ldots, d-k.$
\item[$(3)$] $f(q_i)-f(q_j)-\langle G(q_j), q_i-q_j \rangle \geq 1$ for every $1 \leq i \neq j\leq  d-k.$
\end{lemma}
\begin{proof}
\item[$(1)$] Since $f(q_j) = m(q_j)+1,$ the definition of $m$ leads us to
$$
f(q_j)-f(x)-\langle G(x), q_j-x \rangle = m(q_j)-f(x)-\langle G(x), q_j-x \rangle + 1 \geq 1,
$$
for $x\in E, \: j=1, \ldots, d-k.$
\item[$(2)$] We fix $x \in E$ and $j=1, \ldots, d-k.$ The decomposition of $m$ yields 
$$
m(q_j)=c (P_Y(p_j)+P_Y(T w_j) ) + \langle v, q_j \rangle = c (P_Y(p_j)) + \langle v, q_j \rangle = m(p_j)+ \langle v, q_j-p_j \rangle.
$$
We obtain from this
\begin{align*}
f(x)-f(q_j)& -\langle G(q_j), x-q_j \rangle  = m(x)-m(p_j)+\langle v, p_j-q_j \rangle-\langle G(q_j), x-q_j \rangle -1 \\
& = c\circ(P_Y(x)) + \langle v,x \rangle -m(p_j)+\langle v, p_j-q_j \rangle-\langle v + \varepsilon \alpha w_j , x-q_j \rangle -1 \\
& = c\circ(P_Y(x))  -m(p_j)+\langle v, p_j\rangle-\varepsilon \alpha \langle   w_j , x-q_j \rangle -1 \\
& =  c\circ(P_Y(x))  -m(p_j)+\langle v, p_j\rangle-\varepsilon \alpha \langle   w_j , x-p_j \rangle -\varepsilon \alpha \langle   w_j , p_j-q_j \rangle-1 \\
& =  c\circ(P_Y(x))  -m(p_j)+\langle v, p_j\rangle-\varepsilon \alpha \langle   w_j , x-p_j \rangle +\varepsilon \alpha T-1.
\end{align*}
Now, using Lemma \ref{lemmalipschitzconstantdecomposition}(5), the last term is bigger than or equal to
\begin{align*}
\alpha | P_Y(x)| & + \beta -m(p_j)+\langle v, p_j\rangle-\varepsilon \alpha \langle   w_j , x-p_j \rangle +\varepsilon \alpha T -1 \\
& \geq \alpha | P_Y(x-p_j)| - \alpha | P_Y(p_j)| + \beta -m(p_j)+\langle v, p_j\rangle-\varepsilon \alpha \langle   w_j , x-p_j \rangle +\varepsilon \alpha T -1 \\
& \geq \alpha | P_Y(x-p_j)| -\varepsilon \alpha \langle   w_j , x-p_j \rangle +1,
\end{align*}
where the last inequality follows from the choice of $T.$ Now, since $x\in E,$ the condition $(iii)$ tells us that $x$ does not belong to the cone $V_j,$ which implies that the last term is greater than or equal to
$$
\varepsilon \alpha \langle   w_j , x-p_j \rangle -\varepsilon \alpha \langle   w_j , x-p_j \rangle  +1 =1.
$$
This establishes the inequalities of $(2).$ 
\item[$(3)$] Consider $1 \leq i \neq j\leq d-k.$ Notice that
$$
f(q_i)-f(q_j)=c (P_Y(p_i+T w_i)) - c (P_Y(p_j+T w_j)) + \langle v, q_i-q_j \rangle =  c (P_Y(p_i))-c (P_Y(p_j))+\langle v, q_i-q_j \rangle.
$$
This implies
\begin{align*}
f(q_i)-f(q_j)& -\langle G(q_j), q_i-q_j \rangle = c (P_Y(p_i))-c (P_Y(p_j))+\langle v, q_i-q_j \rangle- \langle v+ \varepsilon \: \alpha \:  w_j, q_i-q_j \rangle \\
& = c (P_Y(p_i))-c (P_Y(p_j))- \varepsilon \: \alpha \langle w_j, p_i-p_j+ T(w_i-w_j) \rangle \\
& = c (P_Y(p_i))-c (P_Y(p_j))- \varepsilon \: \alpha \langle w_j, p_i-p_j \rangle + \varepsilon \: \alpha \: T \left( 1- \langle w_i, w_j \rangle \right) \geq 1,
\end{align*}
where the last inequality follows from the choice of $T.$ 
\end{proof}

\bigskip

\subsection{Properties of the new jet}

We now define the set $E^*=E \cup \lbrace q_1, \ldots, q_{d-k} \rbrace.$ Note that we have already extended the definition of $(f,G)$ to $E^*.$ 

\begin{lemma}\label{lemmapropertiesnewjet}
We have that:
\begin{enumerate}
\item[$(1)$] $X= \textrm{span} \lbrace G(x)-G(y) \: :\: x,y\in E^*\rbrace.$ 
\item[$(2)$] $G$ is continuous on $E^*$ and $f(x) \geq f(y)+\langle G(y),x-y \rangle$ for all $x,y \in E^*.$ 
\item[$(3)$] $|G(x)| \leq 3K$ for every $x\in E^*.$
\item[$(4)$] If $(x_\ell)_\ell, (z_\ell)_\ell$ are sequences in $E^*$ such that $(P_X(x_\ell))_\ell$ is bounded and
$$
\lim_{\ell \to\infty} \left( f(x_\ell)-f(z_\ell)- \langle G(z_\ell), x_\ell-z_\ell \rangle \right) = 0,
$$
then $\lim_{\ell \to\infty} | G(x_\ell)-G(z_\ell) | = 0$.
\end{enumerate}
\end{lemma}
\begin{proof}
\item[$(1)$] By Lemma \ref{lemmalipschitzconstantdecomposition}, there are points $x_1,\ldots, x_k \in E $ with $Y = \textrm{span} \lbrace G(x_j)-v \: : \: j=1, \ldots, k \rbrace,$ where $v$ is that of \eqref{firstdecompositionminimal}. Since the vectors $w_1,\ldots,w_{d-k}$ are linearly independent, the definitions of \eqref{definitionnewjets} show that
$$
\textrm{span} \lbrace G(q_j)-v   \: : \: j=1, \ldots, d-k \rbrace =\textrm{span} \lbrace ( \varepsilon \: \alpha ) w_j   \: : \: j=1, \ldots, d-k \rbrace  = X \cap Y^{\perp}.
$$ 
We thus have that
$$
X= \textrm{span}\lbrace G(x_1)-v, \ldots, G(x_k)-v, G(q_1)-v, \ldots , G(q_{d-k})-v \rbrace.
$$
For every two points $x,y\in E^*,$ we can write
$$
G(x)-G(y)=(G(x)-v)-(G(y)-v),
$$
but notice that $G(z)-v \in Y = \textrm{span} \lbrace G(x_i)-v\rbrace_{i=1}^{k}$ for every $z\in E$ and obviously $G(z)-v \in \textrm{span} \lbrace G(q_j)-v \rbrace_{j=1}^{d-k}$ if $z\in E^* \setminus  E .$ This implies that $G(x)-G(y) \in X$ for every $x,y\in E^*.$ Conversely, if $z\in E^*,$ we can write
$$
G(z)-v=(G(z)-G(x_1))+(G(x_1)-v),
$$
where the first term belongs to $\textrm{span} \lbrace  G(x)-G(y) \: : \: x,y\in E^* \rbrace$ and the second one belongs to $Y=\textrm{span} \lbrace G(x)-G(y)   \: : \: x,y\in E \rbrace.$ We conclude that $X =  \textrm{span} \lbrace G(x)-G(y) \: :\: x,y\in E^*\rbrace.$ 

\item[$(2)$] The points $q_1, \ldots, q_{d-k}$ are distinct and none of them belong to $\overline{E}.$ Because $G$ is continuous on $E, G$ is in fact continuous on $E^*.$ Condition $(i)$ of Theorem \ref{maintheoremlipschitzc1} together with Lemma \ref{lemmadatasarecompatible} tell us that 
$$
f(x) \geq f(y)+ \langle G(y), x-y \rangle , \quad \text{for all} \quad x,y \in E^*.
$$
\item[$(3)$] From \eqref{definitionnewjets}, $G(q_j)=v+ ( \varepsilon \: \alpha ) w_j,$ for $j=1, \ldots, d-k.$ Now Lemma \ref{lemmalipschitzconstantdecomposition} tells us that $|v| \leq K$ and $\alpha \leq 2K,$ where $K$ denotes $\sup_{y\in E}|G(y)|.$ Since $\varepsilon \in (0,1)$ and the vectors $w_j$'s have norm equal to $1,$ we can write $|G(p_j)| \leq |v| + \alpha \leq 3K.$
\item[$(4)$] Suppose that$ (x_\ell)_\ell, (z_\ell)_\ell$ are sequences in $E^*$ such that $(P_X( x_\ell))_k$ is bounded and 
$$
\lim_{\ell\to\infty} \left( f(x_\ell)-f(z_\ell)- \langle G(z_\ell), x_\ell-z_\ell \rangle \right)=0.
$$ In view of Lemma \ref{lemmadatasarecompatible}, it is immediate that there exists $\ell_0$ such that either there is some $1\leq j \leq d-k$ with $x_\ell=z_\ell =q_j$ for all $\ell \geq \ell_0$ or else $x_\ell,z_\ell \in E$ for all $\ell \geq \ell_0.$ In the first case, the conclusion is trivial. In the second case, $\lim_{\ell\to\infty} |G(x_\ell)-G(z_\ell)|=0$ follows from condition $(iv)$ of Theorem \ref{maintheoremlipschitzc1}.
\end{proof}

\bigskip

We now define $m^*(x) = \sup _{y\in E^*} \lbrace f(y)+\langle G(y), x-y \rangle \rbrace$ for every $x\in E^*.$ We already know that $X_{m^{*}}= \textrm{span} \lbrace G(x)-G(y) \: : \: x,y \in E^* \rbrace.$ From Lemma \ref{lemmapropertiesnewjet}, $X_{m^{*}}=X.$ The function $m^*$ is convex and $m^{*}=f$ on $E^*.$ Also, for every $x\in E^*,$ we have that $G(x) \in \partial m^{*}(x)$ and, by virtue of Lemma \ref{lemmapropertiesnewjet}, $m^{*}$ is $3 K$-Lipschitz on $\R^n.$  The function $m^*$ has the decomposition
\begin{equation}\label{seconddecompositionminimal}
m^*= c^* \circ P_X + \langle v^*, \cdot \rangle \quad \text{on} \quad \R^n,
\end{equation}
where $c^*: X \to \R$ is convex and coercive on $X$, and $v^* \in \R^n.$ With the same proof as that of Lemma \ref{lemmalipschitzconstantdecomposition}(2), we see that $v^* \in \partial m^*(z_0)$ for some $z_0\in X,$ the function $c^*$ is $6 K$-Lipschitz and $|v^*| \leq 3 K.$ We study the differentiability of $c^*$ in the following Lemma, which follows from the corresponding result of the general (not necessarily Lipschitz) case.

\begin{lemma} \label{differentiabiltyclosurec^*lipschitz}
The function $c^*$ is differentiable on $\overline{P_X(E^*)}$, and, if $y\in P_X(E^*)$, then $ \nabla c^* (y)= G(x)-v^*$, where $x\in E^*$ is such that $P_X( x) = y_.$
\end{lemma} 

\medskip

\subsection{Construction of the extension}

\begin{lemma}\label{extensionconvexcoercivelipschitz}
Let $h:X \to \R$ be a convex, Lipschitz and coercive function such that $h$ is differentiable on a closed subset $A$ of $X.$ There exists $H \in C^1(X)$ convex,  Lipschitz and coercive such that $H=h$ and $\nabla H = \nabla h$ on $A.$ Moreover, $H$ can be taken so that $\lip(H) \leq M \lip(h),$ where $M=M(n)>0$ is a constant only depending on $n$.
\end{lemma}
\begin{proof}
Since $h$ is convex, its gradient $ \nabla h$ is continuous on $A$. Then, for all $x,y\in A,$ we have
$$
0 \leq \frac{h(x)-h(y)-\langle \nabla h(y), x-y \rangle}{|x-y|} \leq \Big \langle \nabla h(x)-\nabla h(y), \frac{x-y}{|x-y|} \Big \rangle \leq |\nabla h(x)-\nabla h(y)| ,
$$
where the last term tends to $0$ as $|x-y| \to 0$ uniformly on $x,y\in K$ for every compact subset $K$ of $A.$ This shows that the pair $( h, \nabla h)$ defined on $A$ satisfies the conditions of the classical Whitney Extension Theorem for $C^1$ functions. Therefore, there exists a function $\widetilde{h} \in C^1(X)$ such that $\widetilde{h}= h$ and $\nabla \widetilde{h} = \nabla h$ on $A.$ In fact, we can arrange $\lip(\widetilde{h}) \leq \kappa \lip(h),$ where $\kappa=\kappa(n)>0$ is a constant only depending on $n$ (see 
\cite[Claim 2.3]{AzagraMudarra2017PLMS}).
Let us denote $L=\textrm{Lip}(h)$. 

\medskip

For each $\varepsilon>0$, let $\theta_{\varepsilon}:\R\to\R$ be defined by
$$
\theta_{\varepsilon}(t)   =  \left\lbrace
	\begin{array}{ccl}
	0 & \mbox{ if } t\leq 0 \\
	t^2 & \mbox{ if } t\leq\frac{L+\varepsilon}{2} \\
	(L+\varepsilon)\left(t-\frac{L+\varepsilon}{2}\right) +\left(\frac{L+\varepsilon}{2}\right)^2  & \mbox{ if }  t>\frac{L+\varepsilon}{2}
	\end{array}
	\right.
$$
Observe that $\theta_{\varepsilon}\in C^{1}(\R)$, $\textrm{Lip}(\theta_{\varepsilon})=L+\varepsilon$. Now set
$$
\Phi_{\varepsilon}(x)=\theta_{\varepsilon}\left(d(x, A)\right),
$$
where $d(x, A)$ stands for the distance from $x$ to $A$, notice that 
$\Phi_{\varepsilon}(x)=d(x,A)^2$ on an open neighborhood of $A$,
and
 define
$$
H_{\varepsilon}(x)=|\widetilde{h}(x)-h(x)|+ 2\Phi_{\varepsilon}(x).
$$ 
Note that $\textrm{Lip}(\Phi_{\varepsilon})=\textrm{Lip}(\theta_{\varepsilon})$ because $d(\cdot, A)$ is $1$-Lipschitz, and therefore
\begin{equation}
\textrm{Lip}(H_{\varepsilon})\leq \textrm{Lip}(\widetilde{h}) + L+ 2(L+\varepsilon)\leq (3+\kappa)L +2\varepsilon.
\end{equation}
\begin{claim}
$H_{\varepsilon}$ is differentiable on $A$, with $\nabla H_{\varepsilon}(x)=0$ for every $x\in A$.
\end{claim}
\begin{proof}
Same as that of Claim \ref{claimdifferentiabilityphi}.
\end{proof}

Now, because $\Phi_{\varepsilon}$ is continuous and positive on $X \setminus A$, by using mollifiers and a partition of unity , one can construct a function $\varphi_{\varepsilon}\in C^{\infty}(X \setminus A)$ such that
\begin{equation}
|\varphi_{\varepsilon}(x)-H_{\varepsilon}(x)|\leq \Phi_{\varepsilon}(x) \,\,\, \textrm{ for every } x\in X \setminus A,
\end{equation}
and 
\begin{equation}
\textrm{Lip}(\varphi_{\varepsilon})\leq \textrm{Lip}(H_{\varepsilon})+\varepsilon
\end{equation}
(see for instance \cite[Proposition 2.1]{GW} for a proof in the more general setting of Riemannian manifolds, or \cite{AFLR} even for possibly infinite-dimensional Riemannian manifolds).
Let us define $\widetilde{\varphi}=\widetilde{\varphi_{\varepsilon}}:X \to\R$ by
$$
\widetilde{\varphi}(x)=  \left\lbrace
	\begin{array}{ccl}
	\varphi_{\varepsilon}(x) & \mbox{ if } x\in X \setminus A \\
	0  & \mbox{ if }  x\in A.
	\end{array}
	\right.
$$
\begin{claim}
The function $\widetilde{\varphi}$ is differentiable on $X$, and it satisfies $\nabla\widetilde{\varphi}(x_0)=0$ for every $x_0\in A$.
\end{claim}
\begin{proof}
Same as Claim \ref{claimdifferentiabilityvarphi}.
\end{proof}

Note also that
\begin{equation}\label{lip constant of varphi}
\textrm{Lip}(\widetilde{\varphi})=\textrm{Lip}(\varphi_{\varepsilon})\leq \textrm{Lip}(H_{\varepsilon})+\varepsilon\leq (3+\kappa)L +3\varepsilon.
\end{equation}

Next we define 
\begin{equation}
g=g_{\varepsilon}:=\widetilde{h}+\widetilde{\varphi}.
\end{equation}
The function $g$ is differentiable on $X$, and coincides with $h$ on $A$. Moreover, we also have $\nabla g=\nabla h$ on $A$ (because $\nabla\widetilde{\varphi}=0$ on $A$). And, for $x\in X \setminus A$, we have
$$
g(x)\geq \widetilde{h}(x)+H_\varepsilon(x)-\Phi_{\varepsilon}(x)=\widetilde{h}(x)+|h(x)-\widetilde{h}(x)|+\Phi_{\varepsilon}(x)\geq h(x) + \Phi_{\varepsilon}(x).
$$
This shows that $g\geq h,$ which in turn implies that $g$ is coercive. Also, notice that according to \eqref{lip constant of varphi} and the definition of $g$, we have
\begin{equation}
\textrm{Lip}(g)\leq \textrm{Lip}(\widetilde{h}) +\textrm{Lip}(\widetilde{\varphi})\leq \kappa L +(3+\kappa)L +3\varepsilon = ( 3+2\kappa) L+ 3\varepsilon.
\end{equation}

If we define $H=\textrm{conv}(g)$ we thus get that $H$ is convex on $X$ and $F\in C^{1}(X)$, with 
\begin{equation}\label{lispchitz constant of F}
\textrm{Lip}(H)\leq\textrm{Lip}(g)\leq ( 3+2\kappa) L+ 3\varepsilon
\end{equation}
Thus, we can take $\varepsilon$ small enough so that $\textrm{Lip}(H) \leq 2 ( 3+2\kappa) L.$ Finally, we know (by an already familiar argument) that $H=h$ and $\nabla H= \nabla h$ on $A.$ Also, because $h$ is a coercive convex function, we have that $H \geq h$ is also coercive. This completes the proof of Lemma \ref{extensionconvexcoercivelipschitz}.
\end{proof}

\bigskip

Now we are able to finish the proof of Theorem \ref{maintheoremlipschitzc1}. Setting $A: = \overline{P_X(E^*)},$ we see from Lemma \ref{differentiabiltyclosurec^*lipschitz} that $c^*$ is differentiable on $A.$ Moreover, since $c^* :X \to \R$ is convex and coercive on $X,$ Lemma \ref{extensionconvexcoercivelipschitz} provides us with a Lipschitz, convex and coercive function $H$ of class $C^1(X)$ such that $(H, \nabla H)=(c^*, \nabla c^*)$ on $A$ and
$$
\lip(H) \leq M \lip(c^*) \leq 6 M K,
$$ 
{where $M>0$ is a dimensional constant. Recall that $K$ denotes $\sup_{y\in E} |G(y)|.$ We next show that $F:= H \circ P_X + \langle v^*, \cdot \rangle$ is the desired extension of $(f,G).$ Since $H$ is $C^1(X)$ and convex, it is clear that $F$ is $C^1(\R^n)$ and convex as well. Because $H$ is coercive on $X,$ it follows (using Theorem \ref{decomposition theorem}) that $X_F=X.$ Also, since $H(y)=c^*(y)$ for $y\in P_X(E),$ we obtain from \eqref{seconddecompositionminimal} that
$$
F(x)= H(P_X(x))+\langle v^*, x \rangle = c^* (P_X(x))+\langle v^*, x \rangle=m^*(x)=f(x).
$$
Besides, from the second part of Lemma \ref{differentiabiltyclosurec^*lipschitz}, we have, for all $x\in E$, that
$$
\nabla F(x)= \nabla H (P_X(x)) + v^* = G(x)-v^* + v^* = G(x).
$$
Finally, note that 
$$
\lip(F) \leq \lip(H)+ |v^*| \leq 6 M K + 3 K = (6 M + 3) K = (6 M +3 ) \sup_{y\in E} |G(y)|.
$$
The proof of Theorem \ref{maintheoremlipschitzc1} is complete.

\section{Proof of  Theorem \ref{main geometric corollary}}

Let us assume first that there exists such a convex body $W$, and let us check that $N$ and $P=P_X: \R^n \to X$} satisfy conditions $(1)-(4)$. Define $F:\R^n\to\R$ by
$$
F(x)=\theta\left( \mu_{W}(x) \right), \quad x\in \R^n,
$$
where $\theta : \R \to [0, +\infty)$ is a $C^1$ Lipschitz convex function with $\theta(t) = t^2$ whenever $|t| \leq 2$ and $\theta(t)=at$ whenever $|t| \geq 2,$ for a suitable $a>0$. We have that $\partial W=F^{-1}(1)$, and in particular $F=1$ on $E$; besides $$N(x)=\frac{\nabla F(x)}{|\nabla F(x)|} \textrm{ for all } x\in E.
$$
It is clear that $F$ is a Lipschitz convex function of class $C^{1}(\R^n)$. Moreover, by elementary properties of the Minkowski functional and the fact that $\nabla F(0)=0,$ we have
$$
X_F=\textrm{span}\lbrace \nabla F(x) \: : \: x \in \R^n \rbrace =\textrm{span}\lbrace \nabla \mu_W(x) \: : \: x \in \partial W \rbrace= \textrm{span}\lbrace n_W(x)  \: : \: x\in \partial W \rbrace=X.
$$
Therefore $(F, \nabla F)$ satisfies conditions $(i)-(iv)$ of Theorem \ref{maintheoremlipschitzc1} on the set $E^{*}:=E\cup \{0\}$ with projection $P=P_X: \R^n \to X$. Then condition $(1)$ follows directly from $(i)$ (or from the fact that $W$ is convex and $N$ is normal to $\partial W$). In order to check $(2)$, take two sequences $(x_k)_k$, $(z_k)_k$ contained in $E$ with $(P(x_k))_k$ bounded. Now suppose that
$$
\lim_{k\to\infty}\langle N(z_k), x_k-z_k\rangle=0.
$$
Then we also have, using that $F(x_k)=1=F(z_k)$, that
$$
\lim_{k\to\infty}\left( F(x_k)-F(z_k)-\langle \nabla F(z_k), x_k-z_k\rangle \right)=0,
$$ 
and according to $(i)$ of Theorem \ref{maintheoremlipschitzc1} we obtain
\begin{equation}\label{difference of gradients go to 0}
\lim_{k\to\infty}|\nabla F(x_k)- \nabla F(z_k)|=0.
\end{equation}
Suppose, seeking a contradiction that we do not have $\lim_{k\to\infty}| N(x_k)-N(z_k)|=0.$ Then, after possibly passing to subsequences, we may assume that there exists some $\varepsilon>0$ such that
$$
|N(x_k)-N(z_k)|\geq \varepsilon\: \textrm{ for all } k\in\N.
$$
Since $F(x_k)=1, \:F(0)=0$ and $\nabla F(x_k) \in X,$ the convexity of $F$ yields 
$$
0 \leq F(0)-F(x_k)- \langle \nabla F(x_k), -x_k \rangle = -1+\langle \nabla F(x_k), x_k \rangle = -1 + \langle \nabla F(x_k), P(x_k) \rangle
$$
and this shows that $\inf_k |\nabla F(x_k)|>0.$ Thanks to \eqref{difference of gradients go to 0}, we have $\inf_k |\nabla F(z_k)|>0$ too and both $\left( \nabla F(x_k) \right)_k$ and $\left( \nabla F(z_k) \right)_k$ are bounded above because $F$ is Lipschitz. So we may assume, possibly after extracting subsequences again, that $\nabla F(x_k)$ and $\nabla F(z_k)$ converge, respectively, to vectors $\xi, \eta\in\R^n\setminus\{0\}$. By \eqref{difference of gradients go to 0} we then get
$\xi=\eta$, hence also
$$
\varepsilon\leq |N(x_k)-N(z_k)|=\left| \frac{\nabla F(x_k)}{|\nabla F(x_k)|} -\frac{\nabla F(z_k)}{|\nabla F(z_k)|} \right|\to \left| \frac{\xi}{|\xi|}-\frac{\eta}{|\eta|}\right|=0,
$$
a contradiction.

Let us now check $(3)$. Since $0\in\textrm{int}(W)$, we can find $r>0$ such that $B(0, r)\subset W$. Let $y\in E$. If $y$ is parallel to $N(y)$, then $\langle N(y), y\rangle=|y|\geq r.$ Otherwise, by convexity of $W$, the triangle of vertices $0$, $r N(y)$ and $y$, with angles $\alpha, \beta, \gamma$ at those vertices, is contained in $W$. So is the triangle of vertices $0$, $r N(y)$, $p$, where $p$ is the intersection of the line segment $[0, y]$ with the line $L=\{rN(y)+tv : t\in\R\}$, where $v$ is perpendicular to $N(y)$ in the plane $\textrm{span}\{y, N(y)\}$. Then we have that $|p|<|y|$, and $|p|\cos\alpha=r$, hence
$$
\langle N(y), y\rangle=|y|\cos\alpha>|p|\cos\alpha=r>0.
$$
Finally condition $(4)$ follows immediately from $(iii)$ of Theorem \ref{maintheoremlipschitzc1} applied with $E^{*}=E\cup\{0\}$ (and from the fact that $\nabla F (0)=0$).

\medskip

Conversely, assume that $N:E\to\mathbb{S}^{n-1}$ and $P=P_X: \R^n \to X$ satisfy $(1)-(4)$, and let us construct a suitable $W$ with the help of Theorem \ref{maintheoremlipschitzc1}. Choose $r$ such that
\begin{equation}\label{there is an r such that}
0<r<\inf_{y\in E}\langle N(y), y\rangle,
\end{equation}
and define $E^{*}=E\cup\{0\}$, $f:E^{*}\to\R$, $G:E^{*}\to\R^n$ by
$$
f(0)=0, f(x)=1 \textrm{ if } x\in E; \,\,\, G(0)=0, G(x)=\frac{2}{r}N(x) \textrm{ if } x\in E.
$$
It is clear that condition $(3)$ implies that $\textrm{dist}(0, E)>0$, hence the continuity of $G$ on $E^{*}$ is obvious. As for checking that $$f(x)-f(y)-\langle G(y), x-y\rangle\geq 0 \: \textrm{ for all } x, y\in E^{*},$$ the only interesting case is that of $x=0$, $y\in E$, for which we have
$$
f(0)-f(y)-\langle G(y), x-y\rangle=-1+\frac{2}{r}\langle N(y), y\rangle\geq -1+2=1>0.
$$
Therefore condition $(i)$ of Theorem \ref{maintheoremlipschitzc1} is fulfilled. Conditions $(ii)$ and $(iii)$ follow immediately from $(4)$. It only remains for us to check $(iv)$. As before, an a priori less trivial situation consists in taking $x_k=0$, $(z_k)_k\subseteq E$. Note that $(G(z_k))_k$ is always bounded. Assuming that
$$
\lim_{k\to\infty} (f(x_k)-f(z_k)-\langle G(z_k), x_k-z_k\rangle)=0,
$$
we get 
$
\lim_{k\to\infty}\langle G(z_k), z_k\rangle=1,
$
which implies
$$
\lim_{k\to\infty}\langle N(z_k), z_k\rangle=\frac{r}{2},
$$
contradicting \eqref{there is an r such that}. Therefore this situation cannot occur. The rest of cases are immediately dealt with.

Thus we may apply Theorem \ref{maintheoremlipschitzc1} in order to find a convex function $F\in C^{1}(\R^n)$ such that $(F, \nabla F)$ extends the jet $(f,G)$, and $X_F=X$. We then define $W=F^{-1}\left( (-\infty, 1] \right)$. It is easy to check that $W$ is a (possibly unbounded) convex body of class $C^1$ such that $E\subset\partial W$, $0\in\textrm{int}(W)$, $N(x)=n_{W}(x)$ for all $x\in E.$ Moreover, because $F(0)=0$ and $\nabla F(0)=0,$ one can see from the proof of Theorem \ref{maintheoremlipschitzc1} that 
$$
X= \textrm{span} \left( \nabla F(E) \cup \lbrace \nabla F(q_1), \ldots, \nabla F(q_{d-\ell}) \rbrace \right),
$$
where the $q_j$'s are such that $F(q_j) \geq 1$ (see Lemma \ref{lemmadatasarecompatible}). In particular, the $q_j$'s do not belong to $\textrm{int}(W)$ and then $\mu_W(q_j) >0$ for every $j=1, \ldots, d-\ell.$ This implies that
$$
\textrm{span}(n_W(\partial W ))  = \textrm{span} \lbrace \nabla F(x) \: : \: x \in \R^n \setminus \mu_W^{-1}(0) \rbrace \supseteq \textrm{span} \left( \nabla F(E) \cup \lbrace \nabla F(q_1), \ldots, \nabla F(q_{d-\ell}) \rbrace \right) =X.
$$
Since $X_F=X,$ this argument shows that $\textrm{span}(n_W(\partial W ))=X.$
\qed

\medskip

\section*{Acknowledgement}
The authors wish to thank the referee for many suggestions that improved the exposition, and for the statement of the stronger version of Corollary \ref{MainCorollaryWithEssentiallyCoerciveExtensions} included in the final version of this paper.

\medskip


\end{document}